\documentclass{article}
\usepackage[utf8]{inputenc}
\usepackage[colorlinks=true,linktocpage=true,breaklinks=true]{hyperref}
\usepackage{amsmath}
\usepackage{amssymb}
\usepackage{amsthm}
\usepackage[shortlabels]{enumitem}
\usepackage{calc}
\usepackage{geometry}
\usepackage{tikz}
\usepackage[font=small]{caption}

\numberwithin{equation}{section}

\newtheorem{thm}{Theorem}[section]
\newtheorem{ptn}[thm]{Proposition}
\newtheorem{clm}[thm]{Claim}
\newtheorem{cor}[thm]{Corollary}
\newtheorem{lem}[thm]{Lemma}
\newtheorem{pro}[thm]{Problem}
\newtheorem{con}[thm]{Conjecture}

\theoremstyle{definition}
\newtheorem{dfn}[thm]{Definition}

\theoremstyle{remark}
\newtheorem{rmk}[thm]{Remark}

\newcommand{\wt}{\widetilde}
\newcommand{\A}{\mathcal{A}}
\newcommand{\C}{\mathbb{C}}
\newcommand{\D}{\mathcal{D}}
\newcommand{\F}{\mathbb{F}}
\newcommand{\PP}{\mathbb{P}}
\newcommand{\OO}{\mathcal{O}}
\newcommand{\Q}{\mathbb{Q}}
\newcommand{\Z}{\mathbb{Z}}
\newcommand{\cS}{\mathcal{S}}
\DeclareMathOperator{\Aut}{Aut}
\DeclareMathOperator{\Spec}{\mathbf{Spec}}
\DeclareMathOperator{\spec}{Spec}
\DeclareMathOperator{\ord}{ord}
\DeclareMathOperator{\vol}{vol} \DeclareMathOperator{\Supp}{Supp}
\DeclareMathOperator{\tail}{tail} \DeclareMathOperator{\Loc}{Loc}
\DeclareMathOperator{\TV}{\wt{TV}} \DeclareMathOperator{\tv}{TV}
\DeclareMathOperator{\bc}{bc}

\title{The Cheltsov--Rubinstein problem for strongly asymptotically log del Pezzo surfaces\thanks{This is part of the author's PhD work. The author is grateful to his advisor Y. A. Rubinstein for suggesting this problem, and a careful reading and several corrections. The author is also thankful to K. Fujita, I. Cheltsov and K. Zhang for their comments and informing the author of the results in \cite{ACC,Liu}. The research is supported by NSF grants DMS-1906370, 2204347, and University of Maryland Summer Research Fellowship for Summer 2022.}}
\author{Chenzi Jin}
\date{May 2023}

\makeatletter
\def\thanks#1{\protected@xdef\@thanks{\@thanks\protect\footnotetext{#1}}}
\makeatother

\begin{document}

\maketitle

\begin{abstract}

The notion of (strongly) asymptotically log Fano varieties was
introduced in 2013 by Cheltsov--Rubinstein, who posed the problem of
classifying all strongly asymptotically log del Pezzo surfaces with
smooth boundary that admit K\"ahler--Einstein edge metrics. Thanks
to the Cheltsov--Rubinstein classification, this amounts to
considering 10 families. In 8 families the problem has been solved
by work of Cheltsov--Rubinstein, Fujita and Mazzeo--Rubinstein. The
remaining 2 families are rational surfaces parameterized by the
self-intersection of the 0-section $n$ and the number of blow-ups
$m$. By Cheltsov--Rubinstein, Cheltsov--Rubinstein--Zhang and
Fujita, K\"ahler--Einstein edge metrics exist when either $m=0$ or
$m\geq3$ for the first family, and the cases $m=1,2$ have been
studied by Fujita--Liu--S\"u\ss--Zhang--Zhuang and Fujita. The final
remaining family, denoted $\mathrm{(II.6A.n.m)}$ in the
Cheltsov--Rubinstein classification, is more difficult as the
boundary consists of two components, unlike any of the other 9
families. It is the generalization of the football to complex
surfaces with the pair $\mathrm{(II.6A.0.0)}$ being exactly the
football times $\PP^1$. The pairs $\mathrm{(II.6A.n.0)}$ have been
completely understood by the work of Rubinstein--Zhang using the
$\PP^1$-bundle structure of Hirzebruch surfaces. This article
studies the family $\mathrm{(II.6A.n.m)}$ for $m\geq1$. These pairs
no longer have a $\PP^1$-bundle structure and are therefore more
difficult to tackle. The main result is a necessary and sufficient
condition on the angles for the existence of K\"ahler--Einstein edge
metrics, generalizing the Rubinstein--Zhang condition. Thus, we
resolve the Cheltsov--Rubinstein problem for strongly asymptotically
log del Pezzo surfaces.

\end{abstract}

\tableofcontents

\section{Introduction}

The purpose of this article is to complete the solution of the
Cheltsov--Rubinstein problem \cite{CR,Rub} on the classfication of
strongly asymptotically log del Pezzo surfaces admitting a
K\"ahler--Einstein edge metric. This has been a topic of intense
research. See, e.g., \cite{CR,CR2,CRZ,FLSZZ,Rub,Rub2,RZ}. As a
consequence of these works, one (infinite) family of pairs remains
to be understood: the pairs $\mathrm{(II.6A.n.m)}$ (see Definition
\ref{setupdef}). In this article we treat this family completely and
hence solve the Cheltsov--Rubinstein problem for strongly
asymptotically log del Pezzo surfaces. The Cheltsov--Rubinstein
problem for general asymptotically log del Pezzo surfaces is very
challenging and remains wide open.

To explain the problem and our work let us introduce some notation.
Let $X$ be a K\"ahler manifold. Fix a divisor $D=D_1+\cdots+D_r$
with simple normal crossings on $X$, and angle vector
$\beta=(\beta_1,\ldots,\beta_r)\in(0,1]^r$. A K\"ahler--Einstein
edge metric is a smooth K\"ahler--Einstein metric on $\displaystyle
X\setminus\bigcup_iD_i$ that has a conical singularity of angle
$2\pi\beta_i$ along $D_i$ \cite{JMR,Rub,Tia}. The existence of a
K\"ahler--Einstein edge metric with positive Ricci curvature imposes
the cohomological condition that
$$
-K_X-\sum_{i=1}^r\left(1-\beta_i\right)D_i
$$
is ample. When $\beta_i=1$, the metric reduces to the ordinary
(smooth) K\"ahler--Einstein metric. On the other hand, the case when
the $\beta_i$'s are \textit{small} is also of much interest; this
corresponds to the notion of asymptotically log Fano varieties,
introduced by Cheltsov--Rubinstein \cite{CR}:
\begin{dfn}
$(X,D_1+\cdots+D_r)$ is \textit{(strongly) asymptotically log Fano}
if $\displaystyle-K_X-\sum_{i=1}^r\left(1-\beta_i\right)D_i$ is
ample for (all) sufficiently small $\beta\in(0,1]^r$. In dimension
2, these are also called \textit{(strongly) asymptotically log del
Pezzo surfaces}.
\end{dfn}

Strongly asymptotically log del Pezzo surfaces have been classified
by Cheltsov--Rubinstein \cite{CR,Rub2}. A natural question to ask is
the existence of K\"ahler--Einstein edge metrics on such surfaces,
namely the following uniformization problem posed by
Cheltsov--Rubinstein \cite[Problem 9.1]{Rub}.
\begin{pro}\label{mainpro}
Determine which strongly asymptotically log del Pezzo surfaces admit
K\"ahler--Einstein edge metrics for sufficiently small $\beta$.
\end{pro}

Assume the boundary is smooth. By Cheltsov--Rubinstein
classification, we are reduced to the following 10 families.

\begin{thm}
\begin{rm}\cite[Theorems 2.1 and 3.1]{CR}.\end{rm} Let $S$ be a smooth surface, $C_1,\ldots,C_r$ be disjoint smooth irreducible curves on $S$. Then $(S,C_1+\cdots+C_r)$ is strongly asymptotically log del Pezzo if and only if it is one of the following pairs.
\begin{itemize}[leftmargin=!,labelindent=0pt,labelwidth=\widthof{$\mathrm{(II.6A.n.m)}$}]
    \item[$\mathrm{(I.3A)}$] $S\cong\F_1$, and $C_1\in|2(Z_1+F)|$;
    \item[$\mathrm{(I.4A)}$] $S\cong\PP^1\times\PP^1$, and $C_1$ is a smooth elliptic curve of bi-degree $(2,2)$;
    \item[$\mathrm{(I.5.m)}$] $S$ is a blow-up of $\PP^2$ at $0\leq m\leq8$ distinct points on a smooth cubic elliptic curve such that $-K_S$ is ample, and $C_1$ is the proper transform of the elliptic curve;
    \item[$\mathrm{(I.6B.m)}$] $S$ is a blow-up of $\PP^2$ at $m\geq0$ distinct points on a smooth conic, and $C_1$ is the proper transform of the conic;
    \item[$\mathrm{(I.6C.m)}$] $S$ is a blow-up of $\PP^2$ at $m\geq0$ distinct points on a line, and $C_1$ is the proper transform of the line;
    \item[$\mathrm{(I.7.n.m)}$] $S$ is a blow-up of $\F_n$ at $m\geq0$ distinct points on $Z_n$, and $C_1$ is the proper transform of $Z_n$, for any $n\geq0$;
    \item[$\mathrm{(I.8B.m)}$] $S$ is a blow-up of $\F_1$ at $m\geq0$ distinct points on a smooth rational curve in the linear system $|Z_1+F|$, and $C_1$ is the proper transform of the rational curve;
    \item[$\mathrm{(I.9B.m)}$] $S$ is a blow-up of $\PP^1\times\PP^1$ at $m\geq0$ distinct points on a smooth rational curve of bi-degree $(2,1)$ with no two of them on a single curve of bi-degree $(0,1)$, and $C_1$ is the proper transform of the rational curve;
    \item[$\mathrm{(I.9C.m)}$] $S$ is a blow-up of $\PP^1\times\PP^1$ at $m\geq0$ distinct points on a smooth rational curve of bi-degree $(1,1)$, and $C_1$ is the proper transform of the rational curve;
    \item[$\mathrm{(II.6A.n.m)}$] $S$ is a blow-up of $\F_n$ at $m\geq0$ distinct points on $Z_n$ and a smooth rational curve in the linear system $|Z_n+nF|$ with no two of them on a single curve in the linear system $|F|$, $C_1$ is the proper transform of $Z_n$, and $C_2$ is the proper transform of the rational curve, for any $n\geq0$.
\end{itemize}
\end{thm}

As a step towards Problem \ref{mainpro}, Cheltsov--Rubinstein
proposed the following conjecture \cite[Conjecture 1.6]{CR}.
\begin{con}\label{CR}
Let $(S,C)$ be a strongly asymptotically log del Pezzo surface with
$r=1$. Then $(S,C)$ admits a K\"ahler--Einstein edge metric for
sufficiently small $\beta$ if and only if $(K_S+C)^2=0$.
\end{con}

The necessary part and all but 6 sub-families of the sufficient part
were verified. Namely, the pairs
$\mathrm{(I.6B.m),(I.6C.m),(I.7.n.m),(I.8B.m),(I.9C.m)}$ do not
admit K\"ahler--Einstein edge metrics with small angles
\cite[Theorem 1.6]{CR2}, while the pairs
$\mathrm{(I.3A),(I.4A),(I.5.m)}$, and $\mathrm{(I.9B.m)}$ for $m=0$
or $m\geq7$ admit K\"ahler--Einstein edge metrics with small angles
\cite[Propositions 7.4 and 7.5]{CR}, \cite[Theorem 1.3]{CRZ},
\cite[Corollary 1]{JMR}.

Let us now discuss the remaining cases, namely, the sub-families
$\mathrm{(I.9B.m)}$ with $1\leq m\leq6$ and $\mathrm{(II.6A.n.m)}$.

For the sub-families $\mathrm{(I.9B.m)}$ with $1\leq m\leq6$, each
sub-family is parameterized by the position of the $m$ blown-up
points on the smooth $(2,1)$-curve. Let $p_0$ and $p_\infty$ denote
the points where the $(2,1)$-curve on $\PP^1\times\PP^1$ is tangent
to a $(0,1)$-curve. Fujita--Liu--S\"u\ss--Zhang--Zhuang showed that
the pairs $\mathrm{(I.9B.m)}$ contradict the conjecture if $m=1,2$
and exactly 1 of the points $p_0$ and $p_\infty$ is blown-up, by
computing an upper bound for the $\delta$-invariant \cite[Corollary
2.9]{FLSZZ}. Fujita verified that these 2 pairs are the only
counterexamples in the 6 sub-families \cite[Theorem 1.3]{Fuj}. Thus,
Conjecture \ref{CR} of Cheltsov--Rubinstein holds for the vast
majority of the cases, and the remaining cases are semi-stable
degenerations that are now completely understood by the
aforementioned works.

In this article we restrict to the remaining tenth and final family,
$\mathrm{(II.6A.n.m)}$, where $\Supp C$ has multiple components, and
hence is not within the scope of Conjecture \ref{CR}.
\begin{pro}\label{pro}
Determine whether $\mathrm{(II.6A.n.m)}$ admit K\"ahler--Einstein
edge metrics for sufficiently small $\beta$.
\end{pro}

First we introduce some notation to describe the curves appearing in
the study of $\mathrm{(II.6A.n.m)}$.
\begin{dfn}\label{setupdef}
Let $\F_n$ denote the Hirzebruch surface containing a curve $C_1$
whose self intersection number is $-n$, and pick a smooth curve
$C_2$ in $|C_1+nF|$ where $F$ denotes a fiber. Let $S$ be the
blow-up of $\F_n$ at $m$ points on $C_1\cup C_2$ such that no two of
them lie on the same fiber, and $C=\wt{C}_1+\wt{C}_2$, where
$\wt{C}_i\subseteq S$ is the proper transform of $C_i$. This family
is named $\mathrm{(II.6A.n.m)}$. See Figure \ref{pic}.
\end{dfn}

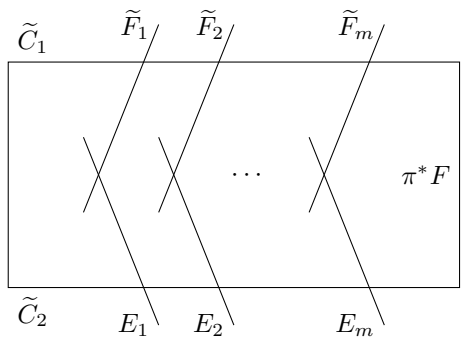
\begin{figure}[!ht]
    \centering
    \begin{tikzpicture}
        \draw(0,0)node[below right]{$\wt{C}_2$}--(6,0)--(6,3)node[midway,left]{$\pi^*F$}--(0,3)node[above right]{$\wt{C}_1$}--cycle;
        \draw(1,1)--(2,3.5)node[left]{$\wt{F}_1$};
        \draw(1,2)--(2,-.5)node[left]{$E_1$};
        \draw(2,1)--(3,3.5)node[left]{$\wt{F}_2$};
        \draw(2,2)--(3,-.5)node[left]{$E_2$};
        \draw(3.2,1.5)node{$\cdots$};
        \draw(4,1)--(5,3.5)node[left]{$\wt{F}_m$};
        \draw(4,2)--(5,-.5)node[left]{$E_m$};
    \end{tikzpicture}
    \caption{The pair $\mathrm{(II.6A.n.m)}$. The surface is given by blowing up $m$ points $p_1,\ldots,p_m$ on the curve $C_2\subseteq\F_n$, where $\F_n$ is the Hirzebruch surface containing a curve $C_1$ whose self intersection number is $-n$ and a generic fiber $F$, and $C_2$ is a smooth element of $|C_1+nF|$. The fiber through the point $p_i$ is denoted by $F_i$, and the exceptional curve over $p_i$ is denoted $E_i$. The curves $\wt{C}_1,\wt{C}_2,\wt{F}_1,\ldots,\wt{F}_m$ are the proper transforms of $C_1,C_2,F_1,\ldots,F_m$, respectively.}\label{pic}
\end{figure}

This family is the generalization of the football to complex
surfaces. In particular, the pair $\mathrm{(II.6A.0.0)}$ can
precisely be equipped with the standard conical metric on the
football times $\PP^1$, if $\beta_1=\beta_2$. Generalizing the
product football metric, using the $\PP^1$-bundle structure of
Hirzebruch surfaces, Rubinstein--Zhang constructed a family of
K\"ahler--Einstein edge metrics using the Calabi ansatz
\cite[Theorem 1.2]{RZ}, and settled the case $m=0$ for all $n$ by
finding an explicit relation between $\beta_1$ and $\beta_2$.
\begin{thm}\label{RZthm}
\begin{rm}\cite[Theorem 1.2]{RZ}.\end{rm} Consider $\mathrm{(II.6A.n.0)}$. Namely, $S$ is the Hirzebruch surface $\F_n$ containing a curve $C_1$ whose self intersection number is $-n$, and $C_2$ is a smooth curve in $|C_1+nF|$ where $F$ denotes a fiber. Let $\beta_1$ and $\beta_2$ be the cone angles along the boundary curves $C_1$ and $C_2$. Then K\"ahler--Einstein edge metrics exist if and only if
\begin{equation}\label{rela}
\beta_1^2-\frac{n}{3}\beta_1^3=\beta_2^2+\frac{n}{3}\beta_2^3.
\end{equation}
\end{thm}

It is difficult to generalize such a construction to the cases
$m\geq1$, as these pairs no longer have a $\PP^1$-bundle structure.
In this article, we study Problem \ref{pro} for $m\geq1$ using
algebraic methods, i.e., verifying log K-stability.

The notion of K-stability introduced by Tian \cite{Tia2} and refined
by Donaldson \cite{Don}, Berman \cite{Ber}, Fujita--Odaka \cite{FO}
and Fujita \cite{Fuj} turned out to characterize the existence of
K\"ahler--Einstein metrics. It was also extended to log setting
\cite{Don2}, namely K\"ahler--Einstein edge metrics exist if and
only if the pair is log K-polystable \cite{BBJ}. Moreover,
Fujita--Odaka \cite{FO} introduced the $\delta$-invariant
(Definition \ref{ddfn}) that characterizes K-stability. In fact, it
has been shown by Blum--Jonsson \cite{BJ} that $\delta>1$ is
equivalent to (uniform) K-stability, which is stronger than
K-polystability, and $\delta\geq1$ is equivalent to K-semistability,
which is weaker than K-polystability. This was also extended to log
setting \cite{Blu}.

The $\delta$-invariant is the main technical tool used in this
article. Our main result generalizes \eqref{rela}, imposing a
relation on the cone angles.

\begin{thm}\label{mainthm}
Consider $\mathrm{(II.6A.n.m)}$ as in Definition \ref{setupdef}. We
have
$$
\delta\left(S,\left(1-\beta_1\right)\wt{C}_1+\left(1-\beta_2\right)\wt{C}_2\right)\leq1,
$$
with equality if and only if
\begin{equation}\label{equalcon}\begin{aligned}
\beta_1^2+\frac{1}{3}\wt{C}_1^2\beta_1^3&=\beta_2^2+\frac{1}{3}\wt{C}_2^2\beta_2^3,&m&\neq1.
\end{aligned}\end{equation}
The pair is log K-polystable, i.e., it admits a K\"ahler--Einstein
edge metric, if and only if \eqref{equalcon} holds.
\end{thm}
The equivalence between stability and existence of
K\"ahler--Einstein edge metrics used in the statement is due to the
resolution of the logarithmic Yau--Tian--Donaldson conjecture
\cite[Theorem 1.6]{LXZ}.

Theorem \ref{mainthm} generalizes Rubinstein--Zhang's Theorem
\ref{RZthm} and completely solves Problem \ref{pro}.

\subsection*{Organization}

We first prove the necessary part of the main theorem by directly
using the definition of the $\delta$-invariant. Section \ref{Pre}
contains the preliminary. We explain the setup in Section \ref{The}.
Then in Sections \ref{Nbv} and \ref{del} we recall the definition of
the $\delta$-invariant and Zariski decomposition which serves as the
main tool in the computation of the $\delta$-invariant. The nef cone
and the negative definite intersection matrices are computed in
Sections \ref{Nef} and \ref{Int}, respectively, as preparation for
the computation of the volumes of divisors. Finally, in Section
\ref{Upp}, we use the computation above to give upper bounds for the
$\delta$-invariant.

Next we make use of the T-variety structure of the pair to give an
alternative proof of the necessary part, as well as a proof of the
sufficient part. In Section \ref{Tva} we recall the relevant
definitions and theorems. The proof is given in Section \ref{Prf}.

\section{Preliminaries}\label{Pre}

\subsection{The pairs (II.6A.n.m)}\label{The}

In this section, we start by explaining the definition of the family
$\mathrm{(II.6A.n.m)}$. As a preparation, we also compute the
intersection numbers of curves as well as the log anti-canonical
divisor of the surface $S$.

For $n\geq0$, let $\F_n$ denote the Hirzebruch surface containing a
curve $C_1$ whose self intersection number is $-n$. A fiber $F$ is a
smooth curve on $\F_n$ whose self intersection number is 0. Pick a
smooth element of $|C_1+nF|$ and call it $C_2$. Let $S$ be the
blow-up of $\F_n$ at $m$ points $p_1,\dots,p_m$ on $C_1\cup C_2$
such that no two of them lie on the same fiber, and
$\wt{C}_i\subseteq S$ the proper transform of $C_i$. Consider the
pair $(S,\wt{C}_1+\wt{C}_2)$. Let $\pi:S\to\F_n$ denote the blow-up
morphism, $E_i$ the exceptional divisor $\pi^{-1}(p_i)$, $F$ a
generic fiber, i.e., a fiber not through any $p_i$, and $F_i$ the
fiber through $p_i$. For an irreducible curve $C\subseteq\F_n$, let
$\wt{C}$ denote its proper transform. For $i=1,2$, define
$$\begin{aligned}
I_i&:=\left\{j\,\middle|\,p_j\in
C_i\right\},&m_i&:=\left|I_i\right|.
\end{aligned}$$

\begin{rmk}
Recall that $\F_{n+1}$ can be obtained by blowing up a point on the
$-n$-curve of $\F_n$ and then blow down the proper transform of the
fiber through that point. In other words, if $S$ is obtained from
$\F_n$ by blowing up $m_1$ points on $C_1\subseteq\F_n$ and $m_2$
points on $C_2\subseteq\F_n$ such that no two of them lie on the
same fiber, then it can also be obtained from $\F_{n+m_1}$ by
blowing up $m_1+m_2$ points on $C_2\subseteq\F_{n+m_1}$. Because of
this, we may assume
$$\begin{aligned}
m_1&=0,&m_2&=m.
\end{aligned}$$
\end{rmk}

The Picard group of $S$ is generated by the linear equivalent
classes of $\pi^*C_2$, $\pi^*F$ and $E_i$, giving intersection
matrix
$$
\begin{pmatrix}
n&1\\
1&0\\
&&-1\\
&&&\ddots\\
&&&&-1
\end{pmatrix}.
$$
Note that
$$\begin{aligned}
\wt{C}_1&\sim\pi^*\left(C_2-nF\right),\\
\wt{C}_2&\sim\pi^*C_2-\sum_iE_i,\\
-K_S&=\wt{C}_1+\wt{C}_2+2\pi^*F\sim\pi^*\left(2C_2+\left(2-n\right)F\right)-\sum_iE_i.
\end{aligned}$$

Let
$$
\Delta_\beta:=\left(1-\beta_1\right)\wt{C}_1+\left(1-\beta_2\right)\wt{C}_2\sim_\Q\pi^*\left(\left(2-\beta_1-\beta_2\right)C_2-\left(n-n\beta_1\right)F\right)-\left(1-\beta_2\right)\sum_iE_i.
$$
Then
\begin{equation}\label{antican}
-K_{S,\Delta_\beta}:=-K_S-\Delta_\beta=\pi^*\left(\left(\beta_1+\beta_2\right)C_2+\left(2-n\beta_1\right)F\right)-\beta_2\sum_iE_i.
\end{equation}

\subsection{Nefness, bigness, volume and Zariski decomposition}\label{Nbv}

In this section, we recall the definition of nefness, bigness, and
volume \cite[Deinition 1.4.1, Definition 2.2.1, Definition 2.2.31
and Remark 2.2.39]{Laz}. We also recall Zariski decomposition, which
we will use as the main tool to compute the volume of a divisor.

\begin{dfn}\label{defalg}
Let $X$ be a projective manifold of dimension $n$ and $D$ a divisor
on $X$. $D$ is called \textit{numerically effective}, or
\textit{nef}, if
$$
D\cdot C\geq0
$$
for any irreducible curve $C$. Define its \textit{volume}
$$
\vol\left(D\right):=\lim_{k\to\infty}\frac{n!}{k^n}\dim
H^0\left(X,kD\right).
$$
In general, if $D$ is a $\Q$-divisor, assume its multiple $mD$ is an
integral divisor. Then its volume is defined by scaling
$$
\vol\left(D\right):=\frac{1}{m^n}\vol\left(mD\right).
$$
We say $D$ is \textit{big} if $\vol(D)>0$.
\end{dfn}

To compute the volume, the following standard result is useful
\cite[Corollary 1.4.41]{Laz}.

\begin{ptn}\label{nefvol}
Let $X$ be a projective manifold of dimension $n$ and $D$ a nef
divisor on $X$. Then
$$
\vol\left(D\right)=D^n.
$$
\end{ptn}

When $X$ is a surface, we can use Zariski decomposition to compute
the volume \cite[Theorems 7.7 and 8.1]{Zar}.
\begin{dfn}\label{ZDecomp}
Let $D$ be an effective divisor on a smooth projective surface $S$.
There is a unique effective divisor $\displaystyle N=\sum_ia_iN_i$
satisfying the following:
\begin{enumerate}[1)]
    \item $D-N$ is nef;
    \item The intersection matrix $(N_i\cdot N_j)$ is negative definite;
    \item For any $i$,
    \begin{equation}\label{PNperp}
        \left(D-N\right)\cdot N_i=0.
    \end{equation}
\end{enumerate}
The divisor $N$ is called the \textit{negative part} of $D$, and the
divisor
$$
P:=D-N
$$
is called the \textit{positive part} of $D$.
\end{dfn}
\begin{ptn}\label{Zvol}
Let $D$ be an effective divisor on a smooth projective surface $S$
with Zariski decomposition $D=P+N$. Then
$$
\vol\left(D\right)=\vol\left(P\right)=P^n.
$$
\end{ptn}

\subsection{\texorpdfstring{$\delta$}{delta}-invariant}\label{del}

In this section we recall the definition of the $\delta$-invariant.

Let $X$ be a projective manifold of dimension $n$, and
$\displaystyle\Delta=\sum_ia_iD_i$ a simple normal crossing
$\Q$-divisor on $X$ with coefficient $a_i\in(0,1)$. Assume
$$
K_{X,\Delta}:=K_X+\Delta
$$
is an anti-ample $\Q$-divisor.

Recall the following definitions \cite[Theorem 3.3]{BBJ}
\cite[Definition 2.24]{KM} .

\begin{dfn}
Let $X$ be a projective manifold. We call $E$ \textit{a prime
divisor over $X$} if there is a projective birational morphism
$f:Y\to X$ with normal variety $Y$, and $E$ is a prime divisor on
$Y$. For a big line bundle $L\to X$, define the
\textit{pseudo-effective threshold}
$$
\tau_L\left(E\right):=\sup\left\{x\in\Q_{>0}\,\middle|\,f^*\left(L\right)-xE\text{
is big}\right\},
$$
and the \textit{expected vanishing order}
$$
S_L\left(E\right):=\frac{1}{\vol\left(L\right)}\int_0^{\tau_L\left(E\right)}\vol\left(f^*\left(L\right)-xE\right)dx.
$$
Define its \textit{log discrepancy}
\begin{equation}\label{Adef}
A_{X,\Delta}\left(E\right):=1-\ord_E\left(f^*\left(K_{X,\Delta}\right)-K_Y\right).
\end{equation}
\end{dfn}

Below is the definition of the $\delta$-invariant \cite[Theorem
C]{BBJ} that we will use in this article.
\begin{dfn}\label{ddfn}
Let $X$ be a projective manifold. For a big line bundle $L\to X$,
define the \textit{$\delta$-invariant}
$$
\delta\left(X,\Delta;L\right):=\inf_{E}\frac{A_{X,\Delta}\left(E\right)}{S_L\left(E\right)},
$$
where $E$ is taken to be prime divisors over $X$. In particular, if
$L=-K_{X,\Delta}$, set
$$
\delta\left(X,\Delta\right):=\delta\left(X,\Delta;-K_{X,\Delta}\right).
$$
\end{dfn}

\section{Nefness criterion}\label{Nef}

For the remainder of the article we let $(S,C)$ be the pair
$\mathrm{(II.6A.n.m)}$. See Section \ref{The}.

To compute Zariski decomposition, we need to check nefness of
various divisors. To simplify this process, in this section we first
compute the nef cone. The result is the following.

\begin{lem}
A divisor
$$
D\sim_\Q\pi^*\left(aC_2+bF\right)-\sum_ic_iE_i
$$
is nef if and only if
\begin{equation}\label{nefconc}\begin{aligned}
0&\leq D\cdot E_i=c_i,\\
0&\leq D\cdot\wt{F}_i=a-c_i,\\
0&\leq D\cdot\wt{C}_1=b,\\
0&\leq D\cdot\wt{C}_2=na+b-\sum_ic_i.
\end{aligned}\end{equation}
\end{lem}
\begin{proof}
Recall Definition \ref{defalg}. If $D$ is nef, then \eqref{nefconc}
is necessary.

On the other hand, if \eqref{nefconc} is satisfied, then the divisor
$$
D':=a\wt{C}_2+\sum_i\left(a-c_i\right)E_i+b\pi^*F\sim_\Q D
$$
is effective. To prove that $D$ is nef, it suffices to show that
$D\cdot C\geq0$ for each irreducible component $C$ of $D'$. Since
\eqref{nefconc} is satisfied,
$$\begin{aligned}
D\cdot\wt{C}_2&=na+b-\sum_ic_i\geq0,\\
D\cdot E_i&=c_i\geq0,\\
D\cdot\pi^*F&=a\geq0.
\end{aligned}$$
This shows $D$ is nef.
\end{proof}

\begin{cor}
The log anti-canonical divisor $-K_{S,\Delta_\beta}$ is ample if and
only if
\begin{equation}\label{ampconc}\begin{aligned}
\beta_1&\in\left(0,\frac{2}{n}\right),&\beta_2&\in\begin{cases}
\left(0,+\infty\right),&n\geq m\\
\left(0,\frac{2}{m-n}\right),&n<m.
\end{cases}
\end{aligned}\end{equation}
\end{cor}
\begin{proof}
Since the ample cone is the interior of the nef cone \cite[Theorem
1.4.23]{Laz}, $-K_{S,\Delta_\beta}$ is ample if and only if the
inequalities of \eqref{nefconc} are strict. By \eqref{antican},
$$\begin{aligned}
\beta_2&>0,\\
\beta_1&>0,\\
2-n\beta_1&>0,\\
2+n\beta_2&>m\beta_2.
\end{aligned}$$
This is \eqref{ampconc}.
\end{proof}

\section{Intersection matrices}\label{Int}

Since the Zariski decomposition of a divisor also involves curves
with negative definite intersection matrices, in this section we
give some collections of curves that have negative definite
intersection matrices. They will be used in the computation of
Zariski decompositions.

\begin{clm}\label{clm}
The following families of divisors have negative definite
intersection matrices.
\begin{enumerate}[(1)]
    \item\label{ndF}$\wt{F}_1,\ldots,\wt{F}_m$;
    \item\label{ndE}$E_1,\ldots,E_m$;
    \item\label{ndC1Fi}$\wt{C}_1,\wt{F}_i$, where $n>1$ and $1\leq i\leq m$;
    \item\label{ndC1F}$\wt{C}_1,\wt{F}_1,\ldots,\wt{F}_m$, where $n>m$;
    \item\label{ndC2Fi}$\wt{C}_2,\wt{F}_i$, where $n<m$ and $1\leq i\leq m$;
    \item\label{ndC1C2Fi}$\wt{C}_1,\wt{C}_2,\wt{F}_i$, where $1<n<m$ and $1\leq i\leq m$;
    \item\label{ndC2FiEj}$\wt{C}_2,\wt{F}_i,E_1,\ldots,\widehat{E}_i,\ldots,E_m$, where $n=0$ and $1\leq i\leq m$.
\end{enumerate}
\end{clm}
\begin{proof}
\ref{ndF} The intersection matrix is $-I$, which is negative
definite.

\ref{ndE} The intersection matrix is $-I$, which is negative
definite.

\ref{ndC1Fi} The intersection matrix is
$$
\begin{pmatrix}
-n&1\\
1&-1
\end{pmatrix}.
$$
Since
$$
\begin{pmatrix}
-n&1\\
1&-1
\end{pmatrix}=\begin{pmatrix}
1&-1\\
&1
\end{pmatrix}\begin{pmatrix}
-\left(n-1\right)\\
&-1
\end{pmatrix}\begin{pmatrix}
1\\
-1&1
\end{pmatrix},
$$
it is negative definite.

\ref{ndC1F} The intersection matrix is
$$
\begin{pmatrix}
-n&1&\cdots&1\\
1&-1\\
\vdots&&\ddots\\
1&&&-1
\end{pmatrix}.
$$
Since
$$
\begin{pmatrix}
-n&1&\cdots&1\\
1&-1\\
\vdots&&\ddots\\
1&&&-1
\end{pmatrix}=\begin{pmatrix}
1&-1&\cdots&-1\\
&1\\
&&\ddots\\
&&&1
\end{pmatrix}\begin{pmatrix}
-\left(n-m\right)\\
&-1\\
&&\ddots\\
&&&-1
\end{pmatrix}\begin{pmatrix}
1\\
-1&1\\
\vdots&&\ddots\\
-1&&&1
\end{pmatrix},
$$
it is negative definite.

\ref{ndC2Fi} The intersection matrix is
$$
\begin{pmatrix}
-\left(m-n\right)\\
&-1
\end{pmatrix},
$$
which is negative definite.

\ref{ndC1C2Fi} The intersection matrix is
$$
\begin{pmatrix}
-n&&1\\
&-\left(m-n\right)\\
1&&-1
\end{pmatrix}.
$$
Since
$$
\begin{pmatrix}
-n&&1\\
&-\left(m-n\right)\\
1&&-1
\end{pmatrix}=\begin{pmatrix}
1&&-1\\
&1\\
&&1
\end{pmatrix}\begin{pmatrix}
-\left(n-1\right)\\
&-\left(m-n\right)\\
&&-1
\end{pmatrix}\begin{pmatrix}
1\\
&1\\
-1&&1
\end{pmatrix},
$$
it is negative definite.

\ref{ndC2FiEj} We may assume $i=1$. The intersection matrix is
$$
\begin{pmatrix}
-m&0&1&\cdots&1\\
0&-1\\
1&&-1\\
\vdots&&&\ddots\\
1&&&&-1
\end{pmatrix}.
$$
Since
\begin{multline*}
\begin{pmatrix}
-m&0&1&\cdots&1\\
0&-1\\
1&&-1\\
\vdots&&&\ddots\\
1&&&&-1
\end{pmatrix}=\\
\begin{pmatrix}
1&0&-1&\cdots&-1\\
&1\\
&&1\\
&&&\ddots\\
&&&&1
\end{pmatrix}\begin{pmatrix}
-1\\
&-1\\
&&-1\\
&&&\ddots\\
&&&&-1
\end{pmatrix}\begin{pmatrix}
1\\
0&1\\
-1&&1\\
\vdots&&&\ddots\\
1&&&&1
\end{pmatrix},
\end{multline*}
it is negative definite.
\end{proof}

\section{Upper bounds for the delta-invariant}\label{Upp}

In this section, we give some upper bounds for the
$\delta$-invariant by computing the ratio
$$
\frac{A_{S,\Delta_\beta}\left(E\right)}{S_L\left(E\right)}
$$
for various divisors $E$. In particular, by letting
$E=\wt{C}_1,\wt{C}_2$ we prove Theorem \ref{C12}, and by letting
$E=E_i$ we prove Theorem \ref{E}. They combine to give the necessary
part of the main theorem.

\begin{lem}\label{C1}
The pseudo-effective threshold
$$
\tau\left(-K_{S,\Delta_\beta},\wt{C}_1\right)=\beta_1+\beta_2,
$$
and
$$
\vol\left(-K_{S,\Delta_\beta}-x\wt{C}_1\right)=\begin{cases}
\vol\left(-K_{S,\Delta_\beta}\right)-2\left(2-n\beta_1\right)x-nx^2,&x\in\left[0,\beta_1\right],\\
\left(\beta_1+\beta_2-x\right)\left(4-\left(m-n\right)\left(x-\beta_1+\beta_2\right)\right),&x\in\left(\beta_1,\beta_1+\beta_2\right].
\end{cases}
$$
\end{lem}
\begin{proof}
By \eqref{nefconc} and \eqref{ampconc},
$$
-K_{S,\Delta_\beta}-x\wt{C}_1=\pi^*\left(\left(\beta_1+\beta_2-x\right)C_2+\left(2-n\beta_1+nx\right)F\right)-\beta_2\sum_iE_i
$$
is nef when $x\in[0,\beta_1]$. Therefore by Proposition
\ref{nefvol},
$$
\vol\left(-K_{S,\Delta_\beta}-x\wt{C}_1\right)=\vol\left(-K_{S,\Delta_\beta}\right)-2\left(2-n\beta_1\right)x-nx^2.
$$

When $x$ reaches $\beta_1$, the intersection number
$(-K_{S,\Delta_\beta}-x\wt{C}_1)\cdot\wt{F}_i$ decreases to 0. When
$x>\beta_1$, assume there is an effective divisor $N$ supported on
$\displaystyle\bigcup_i\wt{F}_i$ that satisfies \eqref{PNperp}. We
can solve
$$
N=\left(x-\beta_1\right)\sum_i\wt{F}_i.
$$
By \eqref{nefconc} and \eqref{ampconc},
$$
-K_{S,\Delta_\beta}-x\wt{C}_1-N=\pi^*\left(\left(\beta_1+\beta_2-x\right)C_2+\left(2+\left(m-n\right)\left(\beta_1-x\right)\right)F\right)-\left(\beta_1+\beta_2-x\right)\sum_iE_i
$$
is nef when $x\in(\beta_1,\beta_1+\beta_2]$. Also note Claim
\ref{clm}\ref{ndF}. By Definition \ref{ZDecomp}, $N$ is indeed the
negative part of $-K_{S,\Delta_\beta}-x\wt{C}_1$ for
$x\in(\beta_1,\beta_1+\beta_2]$. Therefore by Proposition
\ref{Zvol},
$$
\vol\left(-K_{S,\Delta_\beta}-x\wt{C}_1\right)=\left(\beta_1+\beta_2-x\right)\left(4-\left(m-n\right)\left(x-\beta_1+\beta_2\right)\right).
$$

When $x$ reaches $\beta_1+\beta_2$, the volume
$\vol(-K_{S,\Delta_\beta}-x\wt{C}_1)$ decreases to 0. This implies
$$
\tau\left(-K_{S,\Delta_\beta},\wt{C}_1\right)=\beta_1+\beta_2.
$$
\end{proof}

\begin{lem}\label{C1S}
The expected vanishing order
$$
S_{-K_{S,\Delta_\beta}}\left(\wt{C}_1\right)=\beta_1-\frac{2}{\vol\left(-K_{S,\Delta_\beta}\right)}\left(\beta_1^2-\beta_2^2+\frac{1}{3}\left(-n\beta_1^3+\left(m-n\right)\beta_2^3\right)\right).
$$
\end{lem}
\begin{proof}
By Lemma \ref{C1},
$$\begin{aligned}
\int_0^{\beta_1}\vol\left(-K_{S,\Delta_\beta}-x\wt{C}_1\right)dx&=\beta_1\vol\left(-K_{S,\Delta_\beta}\right)-2\beta_1^2+\frac{2}{3}n\beta_1^3,\\
\int_{\beta_1}^{\beta_1+\beta_2}\vol\left(-K_{S,\Delta_\beta}-x\wt{C}_1\right)dx&=2\beta_2^2-\frac{2}{3}\left(m-n\right)\beta_2^3.
\end{aligned}$$
Therefore
$$\begin{aligned}
S_{-K_{S,\Delta_\beta}}\left(\wt{C}_1\right)&=\frac{1}{\vol\left(-K_{S,\Delta_\beta}\right)}\int_0^{\beta_1+\beta_2}\vol\left(-K_{S,\Delta_\beta}-x\wt{C}_1\right)dx\\
&=\beta_1-\frac{2}{\vol\left(-K_{S,\Delta_\beta}\right)}\left(\beta_1^2-\beta_2^2+\frac{1}{3}\left(-n\beta_1^3+\left(m-n\right)\beta_2^3\right)\right).
\end{aligned}$$
\end{proof}

\begin{lem}\label{C2}
The pseudo-effective threshold
$$
\tau\left(-K_{S,\Delta_\beta},\wt{C}_2\right)=\beta_1+\beta_2,
$$
and
$$
\vol\left(-K_{S,\Delta_\beta}-x\wt{C}_2\right)=\begin{cases}
\vol\left(-K_{S,\Delta_\beta}\right)-2\left(2-\left(m-n\right)\beta_2\right)x-\left(m-n\right)x^2,&x\in\left[0,\beta_2\right],\\
\left(\beta_1+\beta_2-x\right)\left(4-n\left(x+\beta_1-\beta_2\right)\right),&x\in\left(\beta_2,\beta_1+\beta_2\right].
\end{cases}
$$
\end{lem}
\begin{proof}
By \eqref{nefconc} and \eqref{ampconc},
$$
-K_{S,\Delta_\beta}-x\wt{C}_2=\pi^*\left(\left(\beta_1+\beta_2-x\right)C_2+\left(2-n\beta_1\right)F\right)-\left(\beta_2-x\right)\sum_iE_i
$$
is nef when $x\in[0,\beta_2]$. Therefore by Proposition
\ref{nefvol},
$$
\vol\left(-K_{S,\Delta_\beta}-x\wt{C}_2\right)=\vol\left(-K_{S,\Delta_\beta}\right)-2\left(2-\left(m-n\right)\beta_2\right)x-\left(m-n\right)x^2.
$$

When $x$ reaches $\beta_2$, the intersection number
$(-K_{S,\Delta_\beta}-x\wt{C}_2)\cdot E_i$ decreases to 0. When
$x>\beta_2$, assume there is an effective divisor $N$ supported on
$\displaystyle\bigcup_iE_i$ that satisfies \eqref{PNperp}. We can
solve
$$
N=\left(x-\beta_2\right)\sum_iE_i.
$$
By \eqref{nefconc} and \eqref{ampconc},
$$
-K_{S,\Delta_\beta}-x\wt{C}_2-N=\pi^*\left(\left(\beta_1+\beta_2-x\right)C_2+\left(2-n\beta_1\right)F\right)
$$
is nef when $x\in(\beta_2,\beta_1+\beta_2]$. Also note Claim
\ref{clm}\ref{ndE}. By Definition \ref{ZDecomp}, $N$ is indeed the
negative part of $-K_{S,\Delta_\beta}-x\wt{C}_2$ for
$x\in(\beta_2,\beta_1+\beta_2]$. Therefore by Proposition
\ref{Zvol},
$$
\vol\left(-K_{S,\Delta_\beta}-x\wt{C}_2\right)=\left(\beta_1+\beta_2-x\right)\left(4-n\left(x+\beta_1-\beta_2\right)\right).
$$

When $x$ reaches $\beta_1+\beta_2$, the volume
$\vol(-K_{S,\Delta_\beta}-x\wt{C}_2)$ decreases to 0. This implies
$$
\tau\left(-K_{S,\Delta_\beta},\wt{C}_2\right)=\beta_1+\beta_2.
$$
\end{proof}

\begin{lem}\label{C2S}
The expected vanishing order
$$
S_{-K_{S,\Delta_\beta}}\left(\wt{C}_2\right)=\beta_2+\frac{2}{\vol\left(-K_{S,\Delta_\beta}\right)}\left(\beta_1^2-\beta_2^2+\frac{1}{3}\left(-n\beta_1^3+\left(m-n\right)\beta_2^3\right)\right).
$$
\end{lem}
\begin{proof}
By Lemma \ref{C2},
$$\begin{aligned}
\int_0^{\beta_2}\vol\left(-K_{S,\Delta_\beta}-x\wt{C}_2\right)dx&=\beta_2\vol\left(-K_{S,\Delta_\beta}\right)-2\beta_2^2+\frac{2}{3}\left(m-n\right)\beta_2^3,\\
\int_{\beta_2}^{\beta_1+\beta_2}\vol\left(-K_{S,\Delta_\beta}-x\wt{C}_2\right)dx&=2\beta_1^2-\frac{2}{3}n\beta_1^3.
\end{aligned}$$
Therefore
$$\begin{aligned}
S_{-K_{S,\Delta_\beta}}\left(\wt{C}_1\right)&=\frac{1}{\vol\left(-K_{S,\Delta_\beta}\right)}\int_0^{\beta_1+\beta_2}\vol\left(-K_{S,\Delta_\beta}-x\wt{C}_2\right)dx\\
&=\beta_2+\frac{2}{\vol\left(-K_{S,\Delta_\beta}\right)}\left(\beta_1^2-\beta_2^2+\frac{1}{3}\left(-n\beta_1^3+\left(m-n\right)\beta_2^3\right)\right).
\end{aligned}$$
\end{proof}

\begin{thm}\label{C12}
If
$$
\beta_1^2-\beta_2^2+\frac{1}{3}\left(-n\beta_1^3+\left(m-n\right)\beta_2^3\right)=0,
$$
then $\delta(S,\Delta_\beta)\leq1$. Otherwise
$\delta(S,\Delta_\beta)<1$.
\end{thm}
\begin{proof}
By \eqref{Adef},
$A_{S,\Delta_\beta}(\wt{C}_i)=1-\ord_{\wt{C}_i}\Delta_\beta=\beta_i$,
for $i=1,2$. Recall Lemmas \ref{C1S} and \ref{C2S},
$$
S_{-K_{S,\Delta_\beta}}\left(\wt{C}_i\right)=\begin{cases}
\beta_1-\frac{2}{\vol\left(-K_{S,\Delta_\beta}\right)}\left(\beta_1^2-\beta_2^2+\frac{1}{3}\left(-n\beta_1^3+\left(m-n\right)\beta_2^3\right)\right),&i=1,\\
\beta_2+\frac{2}{\vol\left(-K_{S,\Delta_\beta}\right)}\left(\beta_1^2-\beta_2^2+\frac{1}{3}\left(-n\beta_1^3+\left(m-n\right)\beta_2^3\right)\right),&i=2.
\end{cases}
$$
By Definition \ref{ddfn},
\begin{equation}\label{dupper}
\delta\left(S,\Delta_\beta\right)\leq\min_{i=1,2}\frac{A_{S,\Delta_\beta}\left(\wt{C}_i\right)}{S_{-K_{S,\Delta_\beta}}\left(\wt{C}_i\right)}.
\end{equation}
When
$$
\beta_1^2-\beta_2^2+\frac{1}{3}\left(-n\beta_1^3+\left(m-n\right)\beta_2^3\right)=0,
$$
\eqref{dupper} becomes $\delta(S,\Delta_\beta)\leq1$. Otherwise if
$$
\beta_1^2-\beta_2^2+\frac{1}{3}\left(-n\beta_1^3+\left(m-n\right)\beta_2^3\right)<0,
$$
then \eqref{dupper} gives
$\delta(S,\Delta_\beta)\leq\frac{A_{S,\Delta_\beta}(\wt{C}_1)}{S_{-K_{S,\Delta_\beta}}(\wt{C}_1)}<1$;
if
$$
\beta_1^2-\beta_2^2+\frac{1}{3}\left(-n\beta_1^3+\left(m-n\right)\beta_2^3\right)>0,
$$
then \eqref{dupper} gives
$\delta(S,\Delta_\beta)\leq\frac{A_{S,\Delta_\beta}(\wt{C}_2)}{S_{-K_{S,\Delta_\beta}}(\wt{C}_2)}<1$.
\end{proof}

Recall that we assume $m\geq1$, i.e., $S$ contains some exceptional
curve $E_i$.
\begin{lem}\label{ngm}
Assume $n\geq m$. The pseudo-effective threshold
$$
\tau\left(-K_{S,\Delta_\beta},E_i\right)=2+\left(n-m\right)\beta_2,
$$
and
\begin{multline*}
\vol\left(-K_{S,\Delta_\beta}-xE_i\right)=\\
\begin{cases}
4\left(\beta_1+\beta_2\right)-n\beta_1^2+\left(n-m\right)\beta_2^2-2\beta_2x-x^2,&x\in\left[0,\beta_1\right],\\
4\left(\beta_1+\beta_2\right)-\left(n-1\right)\beta_1^2+\left(n-m\right)\beta_2^2-2\left(\beta_1+\beta_2\right)x,&x\in\left(\beta_1,2-\left(n-1\right)\beta_1\right],\\
\left(n-1\right)\left(\beta_2+\frac{2-x}{n-1}\right)^2-\left(m-1\right)\beta_2^2,&x\in\left(2-\left(n-1\right)\beta_1,2\right],\\
\frac{\left(2+\left(n-m\right)\beta_2-x\right)^2}{n-m},&x\in\left(2,2+\left(n-m\right)\beta_2\right].
\end{cases}
\end{multline*}
\end{lem}
\begin{proof}
By \eqref{nefconc} and \eqref{ampconc},
$$
-K_{S,\Delta_\beta}-xE_i=\pi^*\left(\left(\beta_1+\beta_2\right)C_2+\left(2-n\beta_1\right)F\right)-\left(\beta_2+x\right)E_i-\sum_{j\neq
i}\beta_2E_j
$$
is nef when $x\in[0,\beta_1]$. Therefore by Proposition
\ref{nefvol},
$$
\vol\left(-K_{S,\Delta_\beta}-xE_i\right)=4\left(\beta_1+\beta_2\right)-n\beta_1^2+\left(n-m\right)\beta_2^2-2\beta_2x-x^2.
$$
When $x$ reaches $\beta_1$, the intersection number
$(-K_{S,\Delta_\beta}-xE_i)\cdot\wt{F}_i$ decreases to 0. When
$x>\beta_1$, assume there is an effective divisor $N$ supported on
$\wt{F}_i$ that satisfies \eqref{PNperp}. We can solve
$$
N=\left(x-\beta_1\right)\wt{F}_i.
$$
By \eqref{nefconc} and \eqref{ampconc},
$$
-K_{S,\Delta_\beta}-xE_i-N=\pi^*\left(\left(\beta_1+\beta_2\right)C_2+\left(2-\left(n-1\right)\beta_1-x\right)F\right)-\left(\beta_1+\beta_2\right)E_i-\sum_{j\neq
i}\beta_2E_j
$$
is nef when $x\in(\beta_1,2-(n-1)\beta_1]$. Also note Claim
\ref{clm}\ref{ndF}. By Definition \ref{ZDecomp}, $N$ is indeed the
negative part of $-K_{S,\Delta_\beta}-xE_i$ for
$x\in(\beta_1,2-(n-1)\beta_1]$. Therefore by Proposition \ref{Zvol},
$$
\vol\left(-K_{S,\Delta_\beta}-xE_i\right)=4\left(\beta_1+\beta_2\right)-\left(n-1\right)\beta_1^2+\left(n-m\right)\beta_2^2-2\left(\beta_1+\beta_2\right)x.
$$

When $x$ reaches $2-(n-1)\beta_1$, the intersection number
$(-K_{S,\Delta_\beta}-xE_i)\cdot\wt{C}_1$ decreases to 0. If $n>1$,
when $x>2-(n-1)\beta_1$, assume there is an effective divisor $N$
supported on $\wt{F}_i\cup\wt{C}_1$ that satisfies \eqref{PNperp}.
We can solve
$$
N=\left(\beta_1-\frac{2-x}{n-1}\right)\wt{C}_1+\frac{nx-2}{n-1}\wt{F}_i.
$$
By \eqref{nefconc} and \eqref{ampconc},
$$
-K_{S,\Delta_\beta}-xE_i-N=\pi^*\left(\left(\beta_2+\frac{2-x}{n-1}\right)C_2\right)-\left(\beta_2+\frac{2-x}{n-1}\right)E_i-\sum_{j\neq
i}\beta_2E_j
$$
is nef when $x\in(2-(n-1)\beta_1,2]$. Also note Claim
\ref{clm}\ref{ndC1Fi}. By Definition \ref{ZDecomp}, $N$ is indeed
the negative part of $-K_{S,\Delta_\beta}-xE_i$ for
$x\in(2-(n-1)\beta_1,2]$. Therefore by Proposition \ref{Zvol},
$$
\vol\left(-K_{S,\Delta_\beta}-xE_i\right)=\left(n-1\right)\left(\beta_2+\frac{2-x}{n-1}\right)^2-\left(m-1\right)\beta_2^2.
$$

When $x$ reaches 2, the intersection number
$(-K_{S,\Delta_\beta}-xE_i)\cdot\wt{F}_j$ decreases to 0 for $j\neq
i$. If $n>m$, when $x>2$, assume there is an effective divisor $N$
supported on $\displaystyle\wt{C}_1\cup\bigcup_j\wt{F}_j$ that
satisfies \eqref{PNperp}. We can solve
$$
N=\left(\beta_1+\frac{x-2}{n-m}\right)\wt{C}_1+x\wt{F}_i+\sum_j\frac{x-2}{n-m}\wt{F}_j.
$$
By \eqref{nefconc} and \eqref{ampconc},
$$
-K_{S,\Delta_\beta}-xE_i-N=\pi^*\left(\left(\beta_2-\frac{x-2}{n-m}\right)C_2\right)-\sum_j\left(\beta_2-\frac{x-2}{n-m}\right)E_i
$$
is nef when $x\in(2,2+(n-m)\beta_2]$. Also note Claim
\ref{clm}\ref{ndC1F}. By Definition \ref{ZDecomp}, $N$ is indeed the
negative part of $-K_{S,\Delta_\beta}-xE_i$ for
$x\in(2,2+(n-m)\beta_2]$. Therefore by Proposition \ref{Zvol},
$$
\vol\left(-K_{S,\Delta_\beta}-xE_i\right)=\frac{\left(2+\left(n-m\right)\beta_2-x\right)^2}{n-m}.
$$

When $x$ reaches $2+(n-m)\beta_2$, the volume
$\vol(-K_{S,\Delta_\beta}-xE_i)$ decreases to 0. This implies
$$
\tau\left(-K_{S,\Delta_\beta},E_i\right)=2+\left(n-m\right)\beta_2.
$$
\end{proof}
\begin{lem}\label{nlm1g2}
Assume $n<m$ and $(n-1)\beta_1\geq(m-n)\beta_2$. The
pseudo-effective threshold
$$
\tau\left(-K_{S,\Delta_\beta},E_i\right)=2,
$$
and
\begin{multline*}
\vol\left(-K_{S,\Delta_\beta}-xE_i\right)=\\
\begin{cases}
4\left(\beta_1+\beta_2\right)-n\beta_1^2+\left(n-m\right)\beta_2^2-2\beta_2x-x^2,&x\in\left[0,\beta_1\right],\\
4\left(\beta_1+\beta_2\right)-\left(n-1\right)\beta_1^2+\left(n-m\right)\beta_2^2-2\left(\beta_1+\beta_2\right)x,&x\in\left(\beta_1,2-\left(n-1\right)\beta_1\right],\\
\left(n-1\right)\left(\beta_2+\frac{2-x}{n-1}\right)^2-\left(m-1\right)\beta_2^2,&x\in\left(2-\left(n-1\right)\beta_1,2-\left(m-n\right)\beta_2\right],\\
\frac{\left(m-1\right)}{\left(m-n\right)\left(n-1\right)}\left(x-2\right)^2,&x\in\left(2-\left(m-n\right)\beta_2,2\right].
\end{cases}
\end{multline*}
\end{lem}
\begin{proof}
By \eqref{nefconc} and \eqref{ampconc},
$$
-K_{S,\Delta_\beta}-xE_i=\pi^*\left(\left(\beta_1+\beta_2\right)C_2+\left(2-n\beta_1\right)F\right)-\left(\beta_2+x\right)E_i-\sum_{j\neq
i}\beta_2E_j
$$
is nef when $x\in[0,\beta_1]$. Therefore by Proposition
\ref{nefvol},
$$
\vol\left(-K_{S,\Delta_\beta}-xE_i\right)=4\left(\beta_1+\beta_2\right)-n\beta_1^2+\left(n-m\right)\beta_2^2-2\beta_2x-x^2.
$$
When $x$ reaches $\beta_1$, the intersection number
$(-K_{S,\Delta_\beta}-xE_i)\cdot\wt{F}_i$ decreases to 0. When
$x>\beta_1$, assume there is an effective divisor $N$ supported on
$\wt{F}_i$ that satisfies \eqref{PNperp}. We can solve
$$
N=\left(x-\beta_1\right)\wt{F}_i.
$$
By \eqref{nefconc} and \eqref{ampconc},
$$
-K_{S,\Delta_\beta}-xE_i-N=\pi^*\left(\left(\beta_1+\beta_2\right)C_2+\left(2-\left(n-1\right)\beta_1-x\right)F\right)-\left(\beta_1+\beta_2\right)E_i-\sum_{j\neq
i}\beta_2E_j
$$
is nef when $x\in(\beta_1,2-(n-1)\beta_1]$. Also note Claim
\ref{clm}\ref{ndF}. By Definition \ref{ZDecomp}, $N$ is indeed the
negative part of $-K_{S,\Delta_\beta}-xE_i$ for
$x\in(\beta_1,2-(n-1)\beta_1]$. Therefore by Proposition \ref{Zvol},
$$
\vol\left(-K_{S,\Delta_\beta}-xE_i\right)=4\left(\beta_1+\beta_2\right)-\left(n-1\right)\beta_1^2+\left(n-m\right)\beta_2^2-2\left(\beta_1+\beta_2\right)x.
$$

When $x$ reaches $2-(n-1)\beta_1$, the intersection number
$(-K_{S,\Delta_\beta}-xE_i)\cdot\wt{C}_1$ decreases to 0. When
$x>2-(n-1)\beta_1$, assume there is an effective divisor $N$
supported on $\wt{F}_i\cup\wt{C}_1$ that satisfies \eqref{PNperp}.
We can solve
$$
N=\left(\beta_1-\frac{2-x}{n-1}\right)\wt{C}_1+\frac{nx-2}{n-1}\wt{F}_i.
$$
By \eqref{nefconc} and \eqref{ampconc},
$$
-K_{S,\Delta_\beta}-xE_i-N=\pi^*\left(\left(\beta_2+\frac{2-x}{n-1}\right)C_2\right)-\left(\beta_2+\frac{2-x}{n-1}\right)E_i-\sum_{j\neq
i}\beta_2E_j
$$
is nef when $x\in(2-(n-1)\beta_1,2-(m-n)\beta_2]$. Also note Claim
\ref{clm}\ref{ndC1Fi}. By Definition \ref{ZDecomp}, $N$ is indeed
the negative part of $-K_{S,\Delta_\beta}-xE_i$ for
$x\in(2-(n-1)\beta_1,2-(m-n)\beta_2]$. Therefore by Proposition
\ref{Zvol},
$$
\vol\left(-K_{S,\Delta_\beta}-xE_i\right)=\left(n-1\right)\left(\beta_2+\frac{2-x}{n-1}\right)^2-\left(m-1\right)\beta_2^2.
$$

When $x$ reaches $2-(m-n)\beta_2$, the intersection number
$(-K_{S,\Delta_\beta}-xE_i)\cdot\wt{C}_2$ decreases to 0. When
$x>2-(m-n)\beta_2$, assume there is an effective divisor $N$
supported on $\wt{F}_i\cup\wt{C}_1\cup\wt{C}_2$ that satisfies
\eqref{PNperp}. We can solve
$$
N=\left(\beta_1-\frac{2-x}{n-1}\right)\wt{C}_1+\frac{nx-2}{n-1}\wt{F}_i+\left(\beta_2+\frac{x-2}{m-n}\right)\wt{C}_2.
$$
By \eqref{nefconc} and \eqref{ampconc},
$$
-K_{S,\Delta_\beta}-xE_i-N=\pi^*\left(\frac{m-1}{\left(n-1\right)\left(m-n\right)}\left(2-x\right)C_2\right)-\frac{m-1}{\left(n-1\right)\left(m-n\right)}\left(2-x\right)E_i-\sum_{j\neq
i}\frac{2-x}{m-n}E_j
$$
is nef when $x\in(2-(m-n)\beta_2,2]$. Also note Claim
\ref{clm}\ref{ndC1C2Fi}. By Definition \ref{ZDecomp}, $N$ is indeed
the negative part of $-K_{S,\Delta_\beta}-xE_i$ for
$x\in(2-(m-n)\beta_2,2]$. Therefore by Proposition \ref{Zvol},
$$
\vol\left(-K_{S,\Delta_\beta}-xE_i\right)=\frac{\left(m-1\right)}{\left(m-n\right)\left(n-1\right)}\left(x-2\right)^2.
$$

When $x$ reaches 2, the volume $\vol(-K_{S,\Delta_\beta}-xE_i)$
decreases to 0. This implies
$$
\tau\left(-K_{S,\Delta_\beta},E_i\right)=2.
$$
\end{proof}
\begin{lem}\label{1l2ng1}
Assume $(n-1)\beta_1<(m-n)\beta_2$ and $n\geq1$. The
pseudo-effective threshold
$$
\tau\left(-K_{S,\Delta_\beta},E_i\right)=2,
$$
and
\begin{multline*}
\vol\left(-K_{S,\Delta_\beta}-xE_i\right)=\\
\begin{cases}
4\left(\beta_1+\beta_2\right)-n\beta_1^2+\left(n-m\right)\beta_2^2-2\beta_2x-x^2,&x\in\left[0,\beta_1\right],\\
4\left(\beta_1+\beta_2\right)-\left(n-1\right)\beta_1^2+\left(n-m\right)\beta_2^2-2\left(\beta_1+\beta_2\right)x,&x\in\left(\beta_1,2-\left(m-n\right)\beta_2\right],\\
\frac{4+\left(m-n\right)\left(4-\left(n-1\right)\beta_1\right)\beta_1-2\left(2+\left(m-n\right)\beta_1\right)x+x^2}{m-n},&x\in\left(2-\left(m-n\right)\beta_2,2-\left(n-1\right)\beta_1\right],\\
\frac{\left(m-1\right)}{\left(m-n\right)\left(n-1\right)}\left(x-2\right)^2,&x\in\left(2-\left(n-1\right)\beta_1,2\right].
\end{cases}
\end{multline*}
\end{lem}
\begin{proof}
By \eqref{nefconc} and \eqref{ampconc},
$$
-K_{S,\Delta_\beta}-xE_i=\pi^*\left(\left(\beta_1+\beta_2\right)C_2+\left(2-n\beta_1\right)F\right)-\left(\beta_2+x\right)E_i-\sum_{j\neq
i}\beta_2E_j
$$
is nef when $x\in[0,\beta_1]$. Therefore by Proposition
\ref{nefvol},
$$
\vol\left(-K_{S,\Delta_\beta}-xE_i\right)=4\left(\beta_1+\beta_2\right)-n\beta_1^2+\left(n-m\right)\beta_2^2-2\beta_2x-x^2.
$$
When $x$ reaches $\beta_1$, the intersection number
$(-K_{S,\Delta_\beta}-xE_i)\cdot\wt{F}_i$ decreases to 0. When
$x>\beta_1$, assume there is an effective divisor $N$ supported on
$\wt{F}_i$ that satisfies \eqref{PNperp}. We can solve
$$
N=\left(x-\beta_1\right)\wt{F}_i.
$$
By \eqref{nefconc} and \eqref{ampconc},
$$
-K_{S,\Delta_\beta}-xE_i-N=\pi^*\left(\left(\beta_1+\beta_2\right)C_2+\left(2-\left(n-1\right)\beta_1-x\right)F\right)-\left(\beta_1+\beta_2\right)E_i-\sum_{j\neq
i}\beta_2E_j
$$
is nef when $x\in(\beta_1,2-(m-n)\beta_2]$. Also note Claim
\ref{clm}\ref{ndF}. By Definition \ref{ZDecomp}, $N$ is indeed the
negative part of $-K_{S,\Delta_\beta}-xE_i$ for
$x\in(\beta_1,2-(m-n)\beta_2]$. Therefore by Proposition \ref{Zvol},
$$
\vol\left(-K_{S,\Delta_\beta}-xE_i\right)=4\left(\beta_1+\beta_2\right)-\left(n-1\right)\beta_1^2+\left(n-m\right)\beta_2^2-2\left(\beta_1+\beta_2\right)x.
$$

When $x$ reaches $2-(m-n)\beta_2$, the intersection number
$(-K_{S,\Delta_\beta}-xE_i)\cdot\wt{C}_2$ decreases to 0. When
$x>2-(m-n)\beta_2$, assume there is an effective divisor $N$
supported on $\wt{F}_i\cup\wt{C}_2$ that satisfies \eqref{PNperp}.
We can solve
$$
N=\left(\beta_2-\frac{2-x}{m-n}\right)\wt{C}_2+\left(x-\beta_1\right)\wt{F}_i.
$$
By \eqref{nefconc} and \eqref{ampconc},
\begin{multline*}
-K_{S,\Delta_\beta}-xE_i-N=\pi^*\left(\left(\beta_1+\frac{2-x}{m-n}\right)C_2+\left(2-\left(n-1\right)\beta_1-x\right)F\right)-\left(\beta_1+\frac{2-x}{m-n}\right)E_i\\
-\sum_{j\neq i}\frac{2-x}{m-n}E_j
\end{multline*}
is nef when $x\in(2-(m-n)\beta_2,2-(n-1)\beta_1]$. Also note Claim
\ref{clm}\ref{ndC2Fi}. By Definition \ref{ZDecomp}, $N$ is indeed
the negative part of $-K_{S,\Delta_\beta}-xE_i$ for
$x\in(2-(m-n)\beta_2,2-(n-1)\beta_1]$. Therefore by Proposition
\ref{Zvol},
$$
\vol\left(-K_{S,\Delta_\beta}-xE_i\right)=\frac{4+\left(m-n\right)\left(4-\left(n-1\right)\beta_1\right)\beta_1-2\left(2+\left(m-n\right)\beta_1\right)x+x^2}{m-n}.
$$

When $x$ reaches $2-(n-1)\beta_1$, the intersection number
$(-K_{S,\Delta_\beta}-xE_i)\cdot\wt{C}_1$ decreases to 0. If $n>1$,
when $x>2-(n-1)\beta_1$, assume there is an effective divisor $N$
supported on $\wt{F}_i\cup\wt{C}_1\cup\wt{C}_2$ that satisfies
\eqref{PNperp}. We can solve
$$
N=\left(\beta_1-\frac{2-x}{n-1}\right)\wt{C}_1+\left(\beta_2-\frac{2-x}{m-n}\right)\wt{C}_2+\frac{nx-2}{n-1}\wt{F}_i.
$$
By \eqref{nefconc} and \eqref{ampconc},
$$
-K_{S,\Delta_\beta}-xE_i-N=\pi^*\left(\frac{m-1}{\left(n-1\right)\left(m-n\right)}\left(2-x\right)C_2\right)-\frac{m-1}{\left(n-1\right)\left(m-n\right)}\left(2-x\right)E_i-\sum_{j\neq
i}\frac{2-x}{m-n}E_j
$$
is nef when $x\in(2-(n-1)\beta_1,2]$. Also note Claim
\ref{clm}\ref{ndC1C2Fi}. By Definition \ref{ZDecomp}, $N$ is indeed
the negative part of $-K_{S,\Delta_\beta}-xE_i$ for
$x\in(2-(n-1)\beta_1,2]$. Therefore by Proposition \ref{Zvol},
$$
\vol\left(-K_{S,\Delta_\beta}-xE_i\right)=\frac{\left(m-1\right)}{\left(m-n\right)\left(n-1\right)}\left(x-2\right)^2.
$$

When $x$ reaches 2, the volume $\vol(-K_{S,\Delta_\beta}-xE_i)$
decreases to 0. This implies
$$
\tau\left(-K_{S,\Delta_\beta},E_i\right)=2.
$$
\end{proof}
\begin{lem}\label{ne0}
Assume $n=0$. The pseudo-effective threshold
$$
\tau\left(-K_{S,\Delta_\beta},E_i\right)=2+\beta_1,
$$
and
$$
\vol\left(-K_{S,\Delta_\beta}-xE_i\right)=\begin{cases}
4\left(\beta_1+\beta_2\right)-m\beta_2^2-2\beta_2x-x^2,&x\in\left[0,\beta_1\right],\\
4\left(\beta_1+\beta_2\right)+\beta_1^2-m\beta_2^2-2\left(\beta_1+\beta_2\right)x,&x\in\left(\beta_1,2-m\beta_2\right],\\
\frac{4+m\left(4+\beta_1\right)\beta_1-2\left(2+m\beta_1\right)x+x^2}{m},&x\in\left(2-m\beta_2,2\right],\\
\left(2+\beta_1-x\right)^2,&x\in\left(2,2+\beta_1\right].
\end{cases}
$$
\end{lem}
\begin{proof}
By \eqref{nefconc} and \eqref{ampconc},
$$
-K_{S,\Delta_\beta}-xE_i=\pi^*\left(\left(\beta_1+\beta_2\right)C_2+2F\right)-\left(\beta_2+x\right)E_i-\sum_{j\neq
i}\beta_2E_j
$$
is nef when $x\in[0,\beta_1]$. Therefore by Proposition
\ref{nefvol},
$$
\vol\left(-K_{S,\Delta_\beta}-xE_i\right)=4\left(\beta_1+\beta_2\right)-m\beta_2^2-2\beta_2x-x^2.
$$
When $x$ reaches $\beta_1$, the intersection number
$(-K_{S,\Delta_\beta}-xE_i)\cdot\wt{F}_i$ decreases to 0. When
$x>\beta_1$, assume there is an effective divisor $N$ supported on
$\wt{F}_i$ that satisfies \eqref{PNperp}. We can solve
$$
N=\left(x-\beta_1\right)\wt{F}_i.
$$
By \eqref{nefconc} and \eqref{ampconc},
$$
-K_{S,\Delta_\beta}-xE_i-N=\pi^*\left(\left(\beta_1+\beta_2\right)C_2+\left(2+\beta_1-x\right)F\right)-\left(\beta_1+\beta_2\right)E_i-\sum_{j\neq
i}\beta_2E_j
$$
is nef when $x\in(\beta_1,2-m\beta_2]$. Also note Claim
\ref{clm}\ref{ndF}. By Definition \ref{ZDecomp}, $N$ is indeed the
negative part of $-K_{S,\Delta_\beta}-xE_i$ for
$x\in(\beta_1,2-m\beta_2]$. Therefore by Proposition \ref{Zvol},
$$
\vol\left(-K_{S,\Delta_\beta}-xE_i\right)=4\left(\beta_1+\beta_2\right)+\beta_1^2-m\beta_2^2-2\left(\beta_1+\beta_2\right)x.
$$

When $x$ reaches $2-m\beta_2$, the intersection number
$(-K_{S,\Delta_\beta}-xE_i)\cdot\wt{C}_2$ decreases to 0. When
$x>2-m\beta_2$, assume there is an effective divisor $N$ supported
on $\wt{F}_i\cup\wt{C}_2$ that satisfies \eqref{PNperp}. We can
solve
$$
N=\left(\beta_2-\frac{2-x}{m}\right)\wt{C}_2+\left(x-\beta_1\right)\wt{F}_i.
$$
By \eqref{nefconc} and \eqref{ampconc},
$$
-K_{S,\Delta_\beta}-xE_i-N=\pi^*\left(\left(\beta_1+\frac{2-x}{m}\right)C_2+\left(2+\beta_1-x\right)F\right)-\left(\beta_1+\frac{2-x}{m}\right)E_i-\sum_{j\neq
i}\frac{2-x}{m}E_j
$$
is nef when $x\in(2-m\beta_2,2]$. Also note Claim
\ref{clm}\ref{ndC2Fi}. By Definition \ref{ZDecomp}, $N$ is indeed
the negative part of $-K_{S,\Delta_\beta}-xE_i$ for
$x\in(2-m\beta_2,2]$. Therefore by Proposition \ref{Zvol},
$$
\vol\left(-K_{S,\Delta_\beta}-xE_i\right)=\frac{4+m\left(4+\beta_1\right)\beta_1-2\left(2+m\beta_1\right)x+x^2}{m}.
$$

When $x$ reaches $2$, the intersection number
$(-K_{S,\Delta_\beta}-xE_i)\cdot E_j$ decreases to 0 for $j\neq i$.
When $x>2$, assume there is an effective divisor $N$ supported on
$\displaystyle\wt{F}_i\cup\wt{C}_2\cup\bigcup_{j\neq i}E_j$ that
satisfies \eqref{PNperp}. We can solve
$$
N=\left(\beta_2+x-2\right)\wt{C}_2+\left(x-\beta_1\right)\wt{F}_i+\sum_{j\neq
i}\left(x-2\right)E_j.
$$
By \eqref{nefconc} and \eqref{ampconc},
$$
-K_{S,\Delta_\beta}-xE_i-N=\pi^*\left(\left(2+\beta_1-x\right)C_2+\left(2+\beta_1-x\right)F\right)-\left(2+\beta_1-x\right)E_i
$$
is nef when $x\in(2,2+\beta_1]$. Also note Claim
\ref{clm}\ref{ndC2FiEj}. By Definition \ref{ZDecomp}, $N$ is indeed
the negative part of $-K_{S,\Delta_\beta}-xE_i$ for
$x\in(2,2+\beta_1]$. Therefore by Proposition \ref{Zvol},
$$
\vol\left(-K_{S,\Delta_\beta}-xE_i\right)=\left(2+\beta_1-x\right)^2.
$$

When $x$ reaches $2+\beta_1$, the volume
$\vol(-K_{S,\Delta_\beta}-xE_i)$ decreases to 0. This implies
$$
\tau\left(-K_{S,\Delta_\beta},E_i\right)=2+\beta_1.
$$
\end{proof}
\begin{lem}\label{ES}
The expected vanishing order
$$
S_{-K_{S,\Delta_\beta}}\left(E_i\right)=1+\frac{1}{\vol\left(-K_{S,\Delta_\beta}\right)}\left(-\left(n-2\right)\beta_1^2+\left(n-m\right)\beta_2^2+\frac{n\left(n-2\right)\beta_1^3+\left(n-m\right)^2\beta_2^3}{3}\right).
$$
\end{lem}
\begin{proof}
By Lemmas \ref{ngm}, \ref{nlm1g2}, \ref{1l2ng1} and \ref{ne0},
$$\begin{aligned}
S_{-K_{S,\Delta_\beta}}\left(E_i\right)&=\frac{1}{\vol\left(-K_{S,\Delta_\beta}\right)}\int_0^{\tau\left(-K_{S,\Delta_\beta},E_i\right)}\vol\left(-K_{S,\Delta_\beta}-xE_i\right)dx\\
&=1+\frac{1}{\vol\left(-K_{S,\Delta_\beta}\right)}\left(-\left(n-2\right)\beta_1^2+\left(n-m\right)\beta_2^2+\frac{n\left(n-2\right)\beta_1^3+\left(n-m\right)^2\beta_2^3}{3}\right).
\end{aligned}$$
\end{proof}
\begin{thm}\label{E}
Assume
$\beta_1^2-\beta_2^2-\frac{1}{3}\left(n\beta_1^3+(n-m)\beta_2^3\right)=0$.
If $m=1$, then $\delta(S,\Delta_\beta)<1$.
\end{thm}
\begin{proof}
By \eqref{Adef},
$A_{S,\Delta_\beta}(E_i)=1-\ord_{E_i}\Delta_\beta=1$. Recall Lemma
\ref{ES},
$$
S_{-K_{S,\Delta_\beta}}\left(E_i\right)=1+\frac{1}{\vol\left(-K_{S,\Delta_\beta}\right)}\left(2-m\right)\beta_1^2\left(1-\frac{n}{3}\beta_1\right).
$$
By definition \ref{ddfn}, when $m=1$,
$$
\delta\left(S,\Delta_\beta\right)\leq\frac{A_{S,\Delta_\beta}\left(E_i\right)}{S_{-K_{S,\Delta_\beta}}\left(E_i\right)}<1.
$$
\end{proof}

Theorems \ref{C12} and \ref{E} combined imply the necessary part of
the main theorem.
\begin{thm}\label{nec}
Consider $\mathrm{(II.6A.n.m)}$ as in Definition \ref{setupdef}. We
have
$$
\delta\left(S,\left(1-\beta_1\right)\wt{C}_1+\left(1-\beta_2\right)\wt{C}_2\right)\leq1,
$$
with equality only if
$$
\beta_1^2+\frac{1}{3}\wt{C}_1^2\beta_1^3=\beta_2^2+\frac{1}{3}\wt{C}_2^2\beta_2^3
$$
and
$$
m\neq1.
$$
\end{thm}
\begin{proof}
By Theorem \ref{C12},
$$
\delta\left(S,\Delta_\beta\right)\leq1.
$$
If equality holds, by Theorem \ref{C12},
$$
\beta_1^2+\frac{1}{3}\wt{C}_1^2\beta_1^3=\beta_2^2+\frac{1}{3}\wt{C}_2^2\beta_1^3;
$$
by Theorem \ref{E},
$$
m\neq1.
$$
\end{proof}

\section{Rational T-varieties of complexity 1}\label{Tva}

In this section we collect some definitions and properties of
rational T-varieties of complexity 1 that will be used in the proof
of the main theorem.

\subsection{Construction of rational T-varieties of complexity 1}

We first recall some definitions for T-varieties \cite[Sections 2.2
and 2.3]{AIP}, in particular, rational T-varieties of complexity 1
\cite[Section 1]{Zha}.
\begin{dfn}
Let $N$ be a lattice and $M$ its dual lattice. We say
$\Delta\subseteq N_\Q:=N\otimes_\Z\Q$ is a \textit{polyhedron} if it
is a finite intersection of half-spaces. Its \textit{tailcone} is
defined by
$$
\tail\left(\Delta\right):=\left\{a\in
N_\Q\,\middle|\,a+\Delta\subseteq\Delta\right\}.
$$
A cone is always assumed to be \textit{pointed}, i.e., it does not
contain non-trivial linear subspaces. Fix a cone $\sigma$, a formal
sum
$$
\D=\sum_{p\in\PP^1}\D_p\otimes p
$$
is called a \textit{polyhedral divisor} on $\PP^1$ with
$\tail(\D):=\sigma$ if each coefficient $\D_p$ is either a
polyhedron with $\tail(\D_p)=\sigma$, or the empty set, and all but
finitely many coefficients are $\sigma$. Define its \textit{locus}
$$
\Loc{\D}:=\PP^1\setminus\bigcup_{\D_p=\varnothing}p
$$
and \textit{degree}
$$
\deg\D:=\sum_{p\in\PP^1}\D_p.
$$
It is called a \textit{p-divisor} if $\deg\D\subsetneqq\sigma$. A
ray $\rho$ of $\sigma$ is called \textit{extremal} if
$\rho\cap\deg\D=\varnothing$. For $u\in\sigma^\vee\cap M$, the
evaluation
$$
\D\left(u\right):=\sum_{p\in\Loc\D}\min_{v\in\D_p}\left\langle
u,v\right\rangle\cdot p
$$
is a $\Q$-divisor on $\Loc\D$. Consider the graded
$\OO_{\Loc\D}$-algebra
$$
\A\left(\D\right):=\bigoplus_{u\in\sigma^\vee\cap
M}\OO_{\Loc\D}\left(\D\left(u\right)\right)\chi^u.
$$
We can define two T-varieties
$$\begin{aligned}
\TV\left(\D\right)&:=\Spec\A\left(\D\right),&\tv\left(\D\right)&:=\spec
H^0\left(\Loc\D,\A\left(\D\right)\right),
\end{aligned}$$
equipped with an action of $T:=\spec\C[M]$ induced by the action of
$\C[M]$ on $\A(\D)$.
\end{dfn}

\begin{thm}
\begin{rm}\cite[Proposition 3.13]{PS}, \cite[Section 1]{Zha}.\end{rm} There are two types of $T$-invariant prime divisors on $\TV(D)$: vertical and horizontal divisors. Vertical divisors are classified by pairs $(p,v)$, where $v$ is a vertex of $\D_p$. Such divisor is denoted $D_{p,v}$. Horizontal divisors are classified by rays $\rho$ of $\sigma$. Such divisor is denoted $D_\rho$. The natural morphism $r:\TV(D)\to\tv(D)$ is a proper birational $T$-equivariant contraction morphism which contracts any horizontal divisor $D_\rho$ where $\rho$ is not extremal. The other divisors survive and are precisely the $T$-invariant prime divisors on $\tv(D)$.
\end{thm}

Similar to the toric case, general T-varieties are obtained by
gluing affine T-varieties.

\begin{dfn}
\begin{rm}\cite[Definition 13]{AIP}, \cite[Definition 1.4]{Zha}.\end{rm}
For two p-divisors $\displaystyle\D=\sum_{p\in\PP^1}\D_p\otimes p$
and $\displaystyle\D'=\sum_{p\in\PP^1}\D'_p\otimes p$, define their
intersection
$$
\D\cap\D':=\sum_{p\in\PP^1}\left(\D_p\cap\D'_p\right)\otimes p.
$$
If $\D'\subseteq\D$, then there is an induced map
$\tv(\D')\to\tv(\D)$. If this map is an open embedding, then $\D'$
is called a \textit{face} of $\D$. A finite set $\cS=\{\D^i\}$ of
p-divisors on $\PP^1$ is called a \textit{divisorial fan} if the
intersection of two p-divisors from $\cS$ is a common face of both
and also belongs to $\cS$. For $p\in\PP^1$, define a \textit{slice}
$\cS_p:=\{\D_p^i\}$. $\cS$ is called \textit{complete} if each slice
is a complete polyhedral subdivision of $N_\Q$.
\end{dfn}

\begin{rmk}
\begin{rm}\cite[Section 4.4]{AIP}.\end{rm} In general, the slices do not capture the entire information of $\cS$, since the information of which polyhedra from different slices belong to the same p-divisor is lost. However, in complexity 1, with the additional information of the degree, this problem can be overcome, which leads to the definition of f-divisors.
\end{rmk}

\begin{dfn}
\begin{rm}\cite[Definition 19]{AIP}.\end{rm} Consider a pair $(\displaystyle\sum_{p\in\PP^1}\cS_p\otimes p,\mathfrak{deg})$ where each $\cS_p$ is a polyhedral subdivision with common tailfan $\Sigma$ and $\mathfrak{deg}\subseteq|\Sigma|$. For any full-dimensional $\sigma\in\tail(\cS)$ that intersects $\mathfrak{deg}$, let $\Delta_p^\sigma$ be the unique polyhedron in $\cS_p$ with tailcone $\sigma$, and $\displaystyle\D^\sigma:=\sum_{p\in\PP^1}\Delta_p^\sigma$. The pair $(\displaystyle\sum_{p\in\PP^1}\cS_p\otimes p,\mathfrak{deg})$ is called an f-divisor if each $\D^\sigma$ is a p-divisor with $\deg\D^\sigma=\mathfrak{deg}\cap\sigma$.
\end{dfn}

\begin{thm}
\begin{rm}\cite[Proposition 20]{AIP}.\end{rm} For any f-divisor $(\displaystyle\sum_{p\in\PP^1}\cS_p\otimes p,\mathfrak{deg})$, there is a divisorial fan $\cS$ with slices $\cS_p$ and degree $\mathfrak{deg}$. Moreover, the T-variety constructed from $\cS$ is independent of the choice of $\cS$.
\end{thm}

Let $X$ be a rational T-variety of complexity 1, characterized by a
complete divisorial fan $\cS$. By \cite[Section 3.1]{PS}, a
T-invariant Cartier divisor is uniquely characterized by its
divisorial support function $\displaystyle
h=\sum_{p\in\PP^1}h_p\otimes p$, where each $h_p$ is defined on the
slice $\cS_p$ and is affine on each polyhedron. Moreover, the linear
part $h_p$ on a polyhedron depends only on its tailcone but not on
$p$. The corresponding Weil divisor is given by the following
theorem.
\begin{thm}\label{divsupp}
\begin{rm}\cite[Corollary 3.19]{PS}.\end{rm} The Weil divisor corresponding to the divisorial support function $h$ is given by
$$
-\sum_\rho
h_t\left(n_\rho\right)D_\rho-\sum_{\left(p,v\right)}\mu\left(v\right)h_p\left(v\right)D_{p,v},
$$
where $h_t$ is the linear part of $h$, $n_\rho$ is the primitive
generator of the extremal ray $\rho$, and $\mu(v)$ is the smallest
natural number such that $\mu(v)v$ is a lattice point.
\end{thm}

\begin{thm}\label{anticantvar}
\begin{rm}\cite[Theorem 3.21]{PS}.\end{rm} Fix an anti-canonical divisor $\displaystyle-K_{\PP^1}=\sum_{p\in\PP^1}a_p\cdot p$. Then the anti-canonical divisor
$$
-K_X=\sum_\rho
D_\rho+\sum_{\left(p,v\right)}\left(\mu\left(v\right)a_p+1-\mu\left(v\right)\right)D_{p,v}.
$$
\end{thm}

\begin{dfn}\label{defdual}
\begin{rm}\cite[Definition 3.22]{PS}.\end{rm} For a divisorial support function $h$ with linear part $h_t$, define
$$
\square_h:=\bigcap_\rho\left\{\left\langle u,n\rho\right\rangle\geq
h_t\left(n_\rho\right)\right\},
$$
where $\rho$ runs through all rays in $\tail(\cS)$. The
\textit{dual} function of $h$ is defined to be
$$
\Psi:=\sum_{p\in\PP^1}\Psi_p\otimes p,
$$
where
$$
\Psi_p\left(u\right):=h_p^*\left(u\right)=\min_v\left(\left\langle
u,v\right\rangle-h_p\left(v\right)\right\}
$$
is the Legendre transform of $h_p$ on $\square_h$. Note that it
suffices to let $v$ run through all vertices in $\cS_p$. We also
define its degree
$$
\deg\Psi:=\sum_{p\in\PP^1}\Psi_p.
$$
\end{dfn}

\subsection{\texorpdfstring{$\delta$}{delta}-invariant and K-stability}

In this section we provide results relevant to the
$\delta$-invariant and K-stability.

\begin{thm}\label{Zhathm}
\begin{rm}\cite[Section 3]{Zha}, \cite[Theorem 3.15]{Sus}.\end{rm} Let $X$ be a rational Fano T-variety of complexity 1 with log terminal singularities, characterized by a complete divisorial fan $\cS$. Let $h$ be the divisorial support function corresponding to $-K_X$ specified by Theorem \ref{anticantvar}, and $\Psi$ its dual function. Define
\begin{equation}\label{defZha}\begin{aligned}
\vol\left(\Psi\right)&:=\int_{\square_h}\deg\Psi\left(u\right)du,\\
\bc\left(\Psi\right)&:=\frac{1}{\vol\left(\Psi\right)}\int_{\square_h}u\cdot\deg\Psi\left(u\right)du,\\
\bc_p&:=\frac{1}{\vol\left(\Psi\right)}\int_{\square_h}\left(\frac{1}{2}\deg\Psi\left(u\right)-\Psi_p\left(u\right)\right)\deg\Psi\left(u\right)du.
\end{aligned}\end{equation}
Then
$$
\delta\left(X\right)=\min\left\{\frac{-h_t\left(n_\rho\right)}{\left\langle\bc\left(\Psi\right),n_\rho\right\rangle-h_t\left(n_\rho\right)}\right\}_\rho\cup\left\{\frac{1}{\mu\left(v\right)\left(\bc_p+\left\langle\bc\left(\Psi\right),v\right\rangle-h_p\left(v\right)\right)}\right\}_{\left(p,v\right)},
$$
where each term corresponds to the valuation given by $D_\rho$ and
$D_{p,v}$.

The Futaki character vanishes precisely when $\bc(\Psi)=0$.
\end{thm}

\begin{thm}\label{Tvarthm}
\begin{rm}\cite[Theorem 1.31]{ACC}, \cite[Theorem 3.2]{Liu}.\end{rm} Let $(X,\Delta)$ be log Fano with Kawamata log terminal singularities. Let $T$ be the maximal torus in $\Aut(X,\Delta)$. Suppose $\dim T=\dim X-1$. Then $(X,\Delta)$ is log K-polystable if and only if the Futaki character vanishes, and for any $T$-invariant prime divisor $E$ on $X$,
$$
\frac{A_{X,\Delta}\left(E\right)}{S_{-K_{X,\Delta}}\left(E\right)}\geq1,
$$
with equality if and only if $E$ is horizontal.
\end{thm}

If $(X,\Delta)$ is toric, recall the following theorem
\cite{Blu,BJ}.
\begin{thm}\label{torvarthm}
Let $(X,\Delta)$ be log Fano and toric. Then it is log K-polystable
if and only if for any torus-invariant prime divisor $E$ on $X$,
$$
\frac{A_{X,\Delta}\left(E\right)}{S_{-K_{X,\Delta}}\left(E\right)}=1.
$$
\end{thm}

\section{Proof of main theorem}\label{Prf}

In this section we make use of the T-variety structure of the pairs
to compute the $\delta$-invariant, independent of the computation
from Section \ref{Upp}, and prove Theorem \ref{mainthm}.

Start from the toric variety $\F_n$. The $\C^*$-action along the
fibers can be lifted to $S$, equipping the surface with the
structure of a T-variety of complexity 1. We choose the f-divisor
$\displaystyle(\sum_{i=0}^m\cS_i\otimes p_i,\mathfrak{deg})$ on
$\PP^1$, where
$$
\cS_i=\begin{cases}
\left\{\left(-\infty,n\right],\left[n,+\infty\right)\right\},&i=0,\\
\left\{\left(-\infty,-1\right],\left[-1,0\right],\left[0,+\infty\right)\right\},&i\geq1,
\end{cases}
$$
(intersections omitted), and $\mathfrak{deg}=\varnothing$, as
explained in Figure \ref{fdiv}.

\begin{figure}[!ht]
    \centering
    \begin{tikzpicture}
        \draw(1,0)node[below left]{$\wt{C}_2$}--(3,0)node[below right]{$\wt{C}_1$}--(3,5)--(1,5)node[midway,above]{$\pi^*F$}--cycle;
        \draw(0,4)node[above right]{$E_1$}--(2,4);
        \draw(0,3)node[above right]{$E_2$}--(2,3);
        \draw(1,2)node[above left]{$\vdots$};
        \draw(0,1)node[above right]{$E_m$}--(2,1);
        \draw[<->](4,5)--(7.5,5)node[above right]{$\cS_0$};
        \filldraw(6.5,5)circle(1pt)node[above]{$n$};
        \draw[<->](4,4)--(7.5,4)node[above right]{$\cS_1$};
        \filldraw(5,4)circle(1pt)node[above]{$-1$};
        \filldraw(5.5,4)circle(1pt)node[above]{$0$};
        \draw[<->](4,3)--(7.5,3)node[above right]{$\cS_2$};
        \filldraw(5,3)circle(1pt)node[above]{$-1$};
        \filldraw(5.5,3)circle(1pt)node[above]{$0$};
        \draw(5.75,2)node[above]{$\vdots$};
        \draw[<->](4,1)--(7.5,1)node[above right]{$\cS_m$};
        \filldraw(5,1)circle(1pt)node[above]{$-1$};
        \filldraw(5.5,1)circle(1pt)node[above]{$0$};
        \draw(9,0)node[below right]{$\PP^1$}--(9,5);
        \filldraw(9,5)circle(1pt)node[above right]{$p_0$};
        \filldraw(9,4)circle(1pt)node[above right]{$p_1$};
        \filldraw(9,3)circle(1pt)node[above right]{$p_2$};
        \draw(9,2)node[above right]{$\vdots$};
        \filldraw(9,1)circle(1pt)node[above right]{$p_m$};
    \end{tikzpicture}
    \caption{An f-divisor corresponding to $S$. The rays $[0,+\infty)$ and $(-\infty,0]$ correspond to the horizontal divisors $\wt{C}_1$ and $\wt{C}_2$, respectively. The vertices correspond to vertical divisors. The vertex $n$ on $\cS_0$ correspond to the fiber over $p_0$. The vertices $-1$ and $0$ on $\cS_i$ correspond to $E_i$ and $\wt{F}_i$, respectively. The vertex $0$ on any trivial slice over $p\in\PP^1$ correspond to the fiber over $p$.}\label{fdiv}
\end{figure}
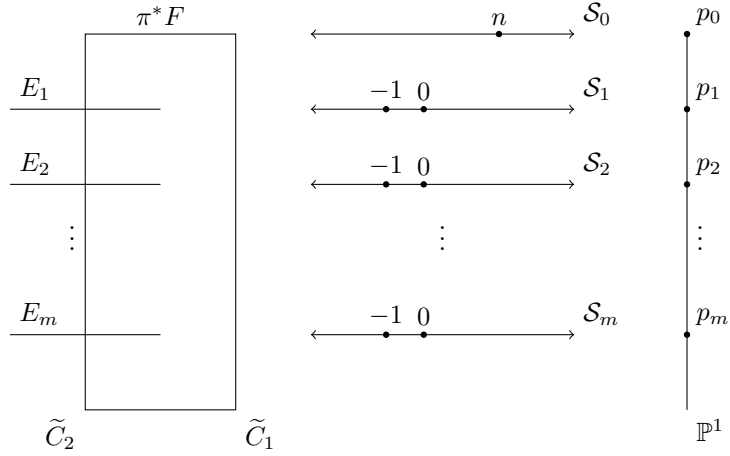

Recall Theorem \ref{anticantvar}. Consider
$$\begin{aligned}
-K_S&=\wt{C}_1+\wt{C}_2+2\pi^*F,\\
-K_{S,\Delta_\beta}&=\beta_1\wt{C}_1+\beta_2\wt{C}_2+2\pi^*F.
\end{aligned}$$
By Theorem \ref{divsupp}, $-K_{S,\Delta_\beta}$ corresponds to the
divisorial support function $h$ with linear part
$$
h_t\left(v\right)=\begin{cases}
\beta_2v,&v<0,\\
-\beta_1v,&v\geq0,
\end{cases}
$$
and satisfies
$$\begin{aligned}
h_{p_0}\left(n\right)&=-2,&h_{p_1}\left(-1\right)=h_{p_1}\left(0\right)=\cdots=h_{p_m}\left(-1\right)=h_{p_m}\left(0\right)&=0,
\end{aligned}$$
as shown in Figure \ref{h}.

\begin{figure}[!ht]
    \centering
    \begin{tikzpicture}
        \draw[<->](0,0)node[above left]{$h_{p_0}(v)$}--(10.5,0)node[right]{$\cS_0$};
        \filldraw(7.5,0)circle(1pt)node[below]{$n$};
        \draw(3.75,0)node[above]{$\beta_2(v-n)-2$};
        \draw(9,0)node[above]{$-\beta_1(v-n)-2$};
        \draw[<->](0,-1.5)node[above left]{$h_{p_i}(v)$}--(10.5,-1.5)node[right]{$\cS_i,i\geq1$};
        \filldraw(3,-1.5)circle(1pt)node[below]{$-1$};
        \filldraw(4.5,-1.5)circle(1pt)node[below]{$0$};
        \draw(1.5,-1.5)node[above]{$\beta_2(v+1)$};
        \draw(3.75,-1.5)node[above]{$0$};
        \draw(7.4,-1.5)node[above]{$-\beta_1v$};
    \end{tikzpicture}
    \caption{The divisorial support function $h$ corresponding to the log anti-canonical divisor $-K_{S,\Delta_\beta}$.}\label{h}
\end{figure}
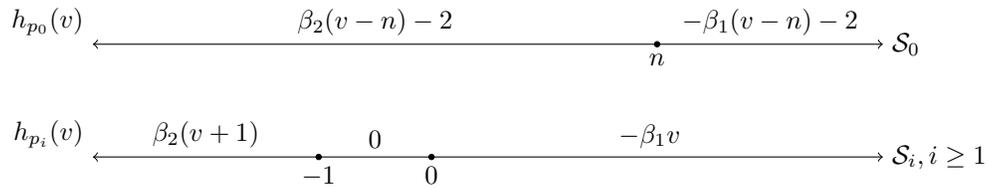

By Definition \ref{defdual}, the dual function $\Psi$ of $h$ is
given by
$$
\Psi_{p_i}\left(u\right)=\begin{cases}
nu+2,&i=0,\\
\min\left\{0,-u\right\},&i\geq1,
\end{cases}
$$
where the domain $\square_h=[-\beta_1,\beta_2]$, and its degree
$$
\deg\Psi\left(u\right)=\begin{cases}
nu+2,&u<0,\\
\left(n-m\right)u+2,&u\geq0.
\end{cases}
$$
By \eqref{defZha}, we compute its volume
$$
\vol\left(\Psi\right)=2\left(\beta_1+\beta_2\right)-\frac{n}{2}\beta_1^2+\frac{n-m}{2}\beta_2^2,
$$
and barycenter
$$\begin{aligned}
\bc\left(\Psi\right)&=-\frac{2}{\vol\left(-K_{S,\Delta_\beta}\right)}\left(\beta_1^2-\beta_2^2-\frac{1}{3}\left(n\beta_1^3+\left(n-m\right)\beta_2^3\right)\right),\\
\bc_P\left(\Psi\right)&=\begin{cases}
-\frac{2}{\vol\left(-K_{S,\Delta_\beta}\right)}\left(\left(2-n\beta_1+n\beta_2\right)\left(\beta_1+\beta_2\right)+\frac{1}{6}\left(n^2\beta_1^3+\left(n^2-m^2\right)\beta_2^3\right)\right),&i=0,\\
\frac{2}{\vol\left(-K_{S,\Delta_\beta}\right)}\left(2\left(\beta_1+\beta_2\right)-n\beta_1^2+\left(n-m+1\right)\beta_2^2+\frac{n^2\beta_1^3+\left(n-m+2\right)\left(n-m\right)\beta_2^3}{6}\right),&i\geq1.
\end{cases}
\end{aligned}$$
Recall that by Theorem \ref{Zhathm}, the Futaki character vanishes
precisely when $\bc(\Psi)=0$, i.e.,
\begin{equation}\label{FutVan}
\beta_1^2-\beta_2^2-\frac{1}{3}\left(n\beta_1^3+\left(n-m\right)\beta_2^3\right)=0.
\end{equation}
We are now ready to compute the $\delta$-invariant.
\begin{equation}\label{ASlist}\begin{aligned}
\frac{A_{S,\Delta_\beta}\left(\wt{C}_1\right)}{S_{-K_{S,\Delta_\beta}}\left(\wt{C}_1\right)}&=\frac{\beta_1}{\beta_1-\frac{2}{\vol\left(-K_{S,\Delta_\beta}\right)}\left(\beta_1^2-\beta_2^2-\frac{1}{3}\left(n\beta_1^3+\left(n-m\right)\beta_2^3\right)\right)},\\
\frac{A_{S,\Delta_\beta}\left(\wt{C}_2\right)}{S_{-K_{S,\Delta_\beta}}\left(\wt{C}_2\right)}&=\frac{\beta_2}{\beta_2+\frac{2}{\vol\left(-K_{S,\Delta_\beta}\right)}\left(\beta_1^2-\beta_2^2-\frac{1}{3}\left(n\beta_1^3+\left(n-m\right)\beta_2^3\right)\right)},\\
\frac{A_{S,\Delta_\beta}\left(E_i\right)}{S_{-K_{S,\Delta_\beta}}\left(E_i\right)}&=\frac{1}{1+\frac{1}{\vol\left(-K_{S,\Delta_\beta}\right)}\left(-\left(n-2\right)\beta_1^2+\left(n-m\right)\beta_2^2+\frac{n\left(n-2\right)\beta_1^3+\left(n-m\right)^2\beta_2^3}{3}\right)},\\
\frac{A_{S,\Delta_\beta}\left(F_i\right)}{S_{-K_{S,\Delta_\beta}}\left(F_i\right)}&=\frac{1}{1+\frac{1}{\vol\left(-K_{S,\Delta_\beta}\right)}\left(-n\beta_1^2+\left(n-m+2\right)\beta_2^2+\frac{n^2\beta_1^3+\left(n-m+2\right)\left(n-m\right)\beta_2^3}{3}\right)}.
\end{aligned}\end{equation}

This reproduces the computation in Section \ref{Upp} and hence
proves Theorem \ref{nec}, the necessary part of Theorem
\ref{mainthm}. We now prove the sufficient part of Theorem
\ref{mainthm}.

\begin{proof}[Proof of Theorem \ref{mainthm}]
Assume
$$
\beta_1^2-\beta_2^2-\frac{1}{3}\left(n\beta_1^3+\left(n-m\right)\beta_2^3\right)=0
$$
and $m\geq 2$. By \eqref{ASlist},
$$\begin{aligned}
\frac{A_{S,\Delta_\beta}\left(\wt{C}_1\right)}{S_{-K_{S,\Delta_\beta}}\left(\wt{C}_1\right)}&=\frac{A_{S,\Delta_\beta}\left(\wt{C}_2\right)}{S_{-K_{S,\Delta_\beta}}\left(\wt{C}_2\right)}=1,\\
\frac{A_{S,\Delta_\beta}\left(E_i\right)}{S_{-K_{S,\Delta_\beta}}\left(E_i\right)}&=\frac{A_{S,\Delta_\beta}\left(F_i\right)}{S_{-K_{S,\Delta_\beta}}\left(F_i\right)}=\frac{1}{1+\frac{1}{\vol\left(-K_{S,\Delta_\beta}\right)}\left(2-m\right)\beta_1^2\left(1-\frac{n}{3}\beta_1\right)}.
\end{aligned}$$

When $m\geq3$,
$$
\frac{A_{S,\Delta_\beta}\left(E_i\right)}{S_{-K_{S,\Delta_\beta}}\left(E_i\right)}=\frac{A_{S,\Delta_\beta}\left(F_i\right)}{S_{-K_{S,\Delta_\beta}}\left(F_i\right)}>1.
$$
Moreover, since $\pi^*F-F_i\geq0$,
$$
\frac{A_{S,\Delta_\beta}\left(\pi^*F\right)}{S_{-K_{S,\Delta_\beta}}\left(\pi^*F\right)}\geq\frac{A_{S,\Delta_\beta}\left(F_i\right)}{S_{-K_{S,\Delta_\beta}}\left(F_i\right)}>1.
$$
By \eqref{FutVan}, the Futaki character vanishes. Therefore by
Theorem \ref{Tvarthm}, the pair $(S,\Delta_\beta)$ is K-polystable.

When $m=2$, for $i=1,2$,
$$
\frac{A_{S,\Delta_\beta}\left(\wt{C}_i\right)}{S_{-K_{S,\Delta_\beta}}\left(\wt{C}_i\right)}=\frac{A_{S,\Delta_\beta}\left(E_i\right)}{S_{-K_{S,\Delta_\beta}}\left(E_i\right)}=\frac{A_{S,\Delta_\beta}\left(F_i\right)}{S_{-K_{S,\Delta_\beta}}\left(F_i\right)}=1.
$$
Therefore by Theorem \ref{torvarthm}, the pair $(S,\Delta_\beta)$ is
K-polystable.
\end{proof}
\begin{rmk}
By \cite[Theorem 5.1]{Fuj2}, the vanishing of the Futaki invariant
also follows directly from the fact that
$A_{S,\Delta_\beta}(\wt{C}_i)=S_{-K_{S,\Delta_\beta}}(\wt{C}_i)$.
\end{rmk}


\begin{thebibliography}{99}

\bibitem{ACC}
C. Araujo, A-M. Castravet, I. Cheltsov, K. Fujita, A-S. Kaloghiros,
J. Martinez-Garcia, C. Shramov, H. S\"u\ss, N. Viswanathan,
\textit{The Calabi problem for Fano threefolds} (2021)
\href{https://www.maths.ed.ac.uk/cheltsov/pdf/Fanos.pdf}{https://www.maths.ed.ac.uk/cheltsov/pdf/Fanos.pdf}

\bibitem{AIP}
K. Altmann, N. O. Ilten, L. Petersen, H. S\"u\ss, R. Vollmert,
\textit{The geometry of T-varieties}, Contributions to algebraic
geometry, EMS Ser. Congr. Rep (2012): 17--69

\bibitem{Ber}
R. Berman, \textit{K-polystability of $\Q$-Fano varieties admitting
K\"ahler--Einstein metrics}, Inventiones mathematicae, 203.3 (2016):
973--1025

\bibitem{BBJ}
R. Berman, S. Boucksom, M. Jonsson, \textit{A variational approach
to the Yau--Tian--Donaldson conjecture}, Journal of the American
Mathematical Society, 34.3 (2021): 605--652

\bibitem{Blu}
H. Blum, \textit{Singularities and K-stability}, Diss. (2018)

\bibitem{BJ}
H. Blum, M. Jonsson, \textit{Thresholds, valuations, and
K-stability}, Advances in Mathematics, 365 (2020): 107062

\bibitem{CR}
I. A. Cheltsov, Y. A. Rubinstein, \textit{Asymptotically log Fano
varieties}, Advances in Mathematics, 285 (2015): 1241--1300

\bibitem{CR2}
I. A. Cheltsov, Y. A. Rubinstein, \textit{On flops and canonical
metrics}, Annali della Scuola Normale Superiore di Pisa. Classe di
scienze, 18.1 (2018): 283--311

\bibitem{CRZ}
I. A. Cheltsov, Y. A. Rubinstein, K. Zhang, \textit{Basis log
canonical thresholds, local intersection estimates, and
asymptotically log del Pezzo surfaces}, Selecta Mathematica, 25.2
(2019): 34

\bibitem{Don}
S. K. Donaldson, \textit{Scalar curvature and stability of toric
varieties}, Journal of Differential Geometry, 62.2 (2002): 289--349

\bibitem{Don2}
S. K. Donaldson, \textit{K\"ahler metrics with cone singularities
along a divisor}, Essays in Mathematics and its Applications: In
Honor of Stephen Smale's 80th Birthday, Berlin, Heidelberg, Springer
(2012): 49--79.

\bibitem{Fuj}
K. Fujita, \textit{Optimal bounds for the volumes of
K\"ahler--Einstein Fano manifolds}, American Journal of Mathematics,
140.2 (2018): 391--414

\bibitem{Fuj2}
K. Fujita, \textit{A valuative criterion for uniform K-stability of
$\Q$-Fano varieties}, Journal f\"ur die reine und angewandte
Mathematik (Crelles Journal), vol. 2019, no. 751 (2019): 309--338

\bibitem{FLSZZ}
K. Fujita, Y. Liu, H. S\"u\ss, K. Zhang, Z. Zhuang, \textit{On the
Cheltsov--Rubinstein conjecture}, preprint (2019) arxiv:1907.02727

\bibitem{FO}
K. Fujita, Y. Odaka, \textit{On the K-stability of Fano varieties
and anticanonical divisors}, Tohoku Mathematical Journal, 70.4
(2018): 511--521

\bibitem{JMR}
T. Jeffres, R. Mazzeo, Y. A. Rubinstein, \textit{K\"ahler--Einstein
metrics with edge singularities}, Annals of Mathematics (2016):
95--176

\bibitem{KM}
J. Koll\'ar, S. Mori, \textit{Birational geometry of algebraic
varieties}, Cambridge Tracts in Mathematics, 134 (1998)

\bibitem{Laz}
R. K. Lazarsfeld, \textit{Positivity in algebraic geometry I:
Classical setting: line bundles and linear series}, Vol. 48,
Springer (2017)

\bibitem{Liu}
Y. Liu, \textit{K-stability of Fano threefolds of rank 2 and degree
14 as double covers} Math. Z. 303, 38 (2023)

\bibitem{LXZ}
Y. Liu, C. Xu, Z. Zhuang, \textit{Finite generation for valuations
computing stability thresholds and applications to K-stability},
Annals of Mathematics, 196.2 (2022): 507--566


\bibitem{PS}
L. Petersen, H. S\"u\ss, \textit{Torus invariant divisors}, Israel
Journal of Mathematics, 182 (2011): 481--504.

\bibitem{Rub}
Y. A. Rubinstein, \textit{Smooth and singular K\"ahler--Einstein
metrics}, Geometric and spectral analysis, 630 (2014): 45--138

\bibitem{Rub2}
Y. A. Rubinstein, \textit{Classification of strongly asymptotically
log del Pezzo flags and surfaces}, International Journal of
Mathematics, 33.10n11 (2022): 2250077

\bibitem{RZ}
Y. A. Rubinstein, K. Zhang, \textit{Angle deformation of
K\"ahler--Einstein edge metrics on Hirzebruch surfaces}, Pure and
Applied Mathematics Quarterly, 18.1 (2022): 343--366

\bibitem{Sus}
H. S\"u\ss, \textit{Fano threefolds with 2-torus action}, Documenta
Mathematica, 19 (2014): 905--940

\bibitem{Tia}
G. Tian, \textit{K\"ahler--Einstein metrics on algebraic manifolds},
Transcendental methods in algebraic geometry (Cetraro, 1994),
Lecture Notes in Math., Vol. 1646, Springer (1996): 143--185

\bibitem{Tia2}
G. Tian, \textit{K\"ahler--Einstein metrics with positive scalar
curvature}, Inventiones mathematicae, 130.1 (1997): 1--37

\bibitem{Zar}
O. Zariski, \textit{The theorem of Riemann-Roch for high multiples
of an effective divisor on an algebraic surface}, Annals of
Mathematics, 76, no. 3 (1962): 560--615

\bibitem{Zha}
K. Zhang, \textit{Stability thresholds on rational Fano T-varieties
of complexity one},
\href{https://keweizhang666.github.io/kwzhang.github.io/Stability_thresholds_on_Fano_T_varieties-2.pdf}{https://}
\href{https://keweizhang666.github.io/kwzhang.github.io/Stability_thresholds_on_Fano_T_varieties-2.pdf}{keweizhang666.github.io/kwzhang.github.io/Stability\_thresholds\_on\_Fano\_T\_varieties-2.pdf}

\end{thebibliography}
\end{document}